\documentclass[a4paper,12pt]{amsart}

\usepackage[utf8]{inputenc}
\usepackage{amsfonts,amssymb}
\usepackage{amsmath}
\usepackage{graphicx}

\usepackage{comment}
\usepackage{color}
\usepackage{tikz}
\usetikzlibrary{topaths}

\usepackage{a4wide}

\usepackage{empheq}
\numberwithin{equation}{section}

\usepackage{amsthm}
\theoremstyle{plain}
\newtheorem*{mainresult}{Theorem A}
\newtheorem{theorem}{Theorem}[section]
\newtheorem{lemma}[theorem]{Lemma}
\newtheorem{corollary}[theorem]{Corollary}
\newtheorem{proposition}[theorem]{Proposition}
\newtheorem{observation}[theorem]{Observation}
\newtheorem{cit}[theorem]{Citation}
\newtheorem{conjecture}[theorem]{Conjecture}
\newtheorem*{lemma*}{Lemma}
\newtheorem*{corollary*}{Corollary}

\theoremstyle{definition}
\newtheorem{definition}[theorem]{Definition}
\newtheorem{remark}[theorem]{Remark}

\newcommand{\R}{\mathbb{R}}
\newcommand{\N}{\mathbb{N}}
\newcommand{\Z}{\mathbb{Z}}

\newcommand{\calP}{\mathcal{P}}

\newcommand{\defeq}{\mathrel{\mathop{:}}=}

\newcommand{\abs}[1]{\lvert #1 \rvert}
\newcommand{\gen}[1]{\langle #1 \rangle}

\newcommand{\lk}{\operatorname{lk}}
\newcommand{\st}{\operatorname{st}}
\newcommand{\Hom}{\operatorname{Hom}}
\newcommand{\CAT}{\operatorname{CAT}}
\newcommand{\F}{\operatorname{F}}
\newcommand{\FP}{\operatorname{FP}}
\newcommand{\id}{\operatorname{id}}

\newcommand{\match}{\mathcal{M}}

\newcommand{\elsplit}{\Lambda}
\newcommand{\elmerge}{\mathrm{V}}
\newcommand{\elnothing}{\mathrm{I}}

\DeclareMathOperator{\image}{image}

\numberwithin{equation}{section}

\begin{document}

\title[$\Sigma^m(F_n)$ and $\Sigma^m(H_n)$]{On the $\Sigma$-invariants of generalized Thompson groups and Houghton groups}

\date{\today}


\author{Matthew C.~B.~Zaremsky}
\address{Department of Mathematical Sciences, Binghamton University, Binghamton, NY 13902}
\email{zaremsky@math.binghamton.edu}

\begin{abstract}
 We compute the higher $\Sigma$-invariants $\Sigma^m(F_{n,\infty})$ of the generalized Thompson groups $F_{n,\infty}$, for all $m,n\ge 2$. This extends the $n=2$ case done by Bieri, Geoghegan and Kochloukova, and the $m=2$ case done by Kochloukova. Our approach differs from those used in the $n=2$ and $m=2$ cases; we look at the action of $F_{n,\infty}$ on a $\CAT(0)$ cube complex, and use Morse theory to compute all the $\Sigma^m(F_{n,\infty})$.
 
 We also obtain lower bounds on $\Sigma^m(H_n)$, for the Houghton groups $H_n$, again using actions on $\CAT(0)$ cube complexes, and discuss evidence that these bounds are sharp.
\end{abstract}

\maketitle
\thispagestyle{empty}


\section*{Introduction}

A group is of \emph{type $\F_m$} if it has a classifying space with compact $m$-skeleton. These \emph{finiteness properties} of groups are natural generalizations of finite generation ($\F_1$) and finite presentability ($\F_2$). In 1987 and 1988, Bieri, Neumann, Strebel and Renz introduced a family of geometric invariants $\Sigma^m(G)$ ($m\in\N$), defined whenever $G$ is of type $\F_m$, which reveal a wealth of information about $G$ and $\Hom(G,\R)$. However, since the $\Sigma^m(G)$ contain so much information, e.g., they serve as a complete catalog of precisely which subgroups of $G$ containing $[G,G]$ have which finiteness properties, they are in general quite difficult to compute.

Thanks to this difficulty, there are very few groups whose higher $\Sigma$-invariants are completely known. 
If $\Hom(G,\R)$ is trivial then all $\Sigma^m(G)$ are empty, so in that case the question is uninteresting, e.g., for groups with finite abelianization. Focusing on groups for which $\Hom(G,\R)$ is sufficiently large, the only really robust family of groups for which the question of all the higher $\Sigma$-invariants is 100\% solved is the family of right-angled Artin groups, done independently by Bux--Gonzalez \cite{bux99} and Meier--Meinert--VanWyk \cite{meier98}. Other interesting families of groups for which there are substantial partial results about the higher $\Sigma$-invariants include Artin groups \cite{meier01}, solvable $S$-arithmetic groups \cite{bux04}, and metabelian groups \cite{meinert96,meinert97,kochloukova99}. The question of the higher $\Sigma$-invariants of a direct product, in terms of the invariants of the factors, is also solved \cite{bieri10a}.

The generalized Thompson groups $F_{n,\infty}$ ($n\ge2$), \textbf{which we will just denote by $F_n$ from now on}, can be quickly defined by their standard presentations
$$F_n\cong\gen{x_i~(i\in\N_0)\mid x_j x_i = x_i x_{j+(n-1)} \text{ for all }i<j}\text{.}$$
These groups were first introduced by Brown in \cite{brown87} as an ``$F$-like'' version of the Higman--Thompson groups $V_{n,r}$. They generalize Thompson's group $F$, namely $F=F_2$. The $F_n$ are all of type $\F_\infty$ \cite{brown87}. The group $F_n$ can also be described as the group of orientation preserving piecewise linear self homeomorphisms of $[0,1]$ with slopes powers of $n$ and breakpoints in $\Z[1/n]$. These groups are interesting for many reasons; from the perspective of $\Sigma$-invariants they are interesting for instance since every proper quotient of $F_n$ is abelian \cite{brown87,brin98}, and so the $\Sigma$-invariants reveal the finiteness properties of \emph{every} normal subgroup of $F_n$. Also, $F_n$ abelianizes to $\Z^n$, and so homomorphisms to $\R$ become more and more prevalent as $n$ goes up. In contrast, the ``type $V$'' Higman--Thompson groups $V_{n,r}$ are virtually simple \cite{higman74}, so have no non-trivial maps to $\R$ (and their $\Sigma$-invariants are empty).

The main result of the present work is a complete computation of $\Sigma^m(F_n)$ for all relevant $m$ and $n$. The previously known results are as follows. First, $\Sigma^1(F_2)$ was computed in the original Bieri--Neumann--Strebel paper \cite{bieri87}. In \cite{bieri10}, Bieri, Geoghegan and Kochloukova computed $\Sigma^m(F_2)$ for all $m$. In the other ``variable'', $n$, Kochloukova computed $\Sigma^2(F_n)$ for all $n$ in \cite{kochloukova12}. The techniques used there however proved difficult to extend to the cases when $n$ and $m$ are both greater than $2$. Our approach differs from those in \cite{bieri10} and \cite{kochloukova12}. We look at the action of $F_n$ on a proper $\CAT(0)$ cube complex $X_n$, and use topological and combinatorial tools to compute all the $\Sigma^m(F_n)$. This builds off work of the author and Witzel, in \cite{witzel15}, where the $\Sigma^m(F_2)$ computations from \cite{bieri10} were redone using such an action of $F=F_2$.

Taking Kochloukova's computation of $\Sigma^2(F_n)$ for granted, our main result can be phrased succinctly as:

\begin{mainresult}
 For any $n,m \ge 2$, we have $\Sigma^m(F_n)=\Sigma^2(F_n)$.
\end{mainresult}

Note that for any group $G$ of type $\F_\infty$ one always has
$$\Sigma^1(G)\supseteq \Sigma^2(G) \supseteq \cdots \supseteq \Sigma^\infty(G) \text{.}$$

A more detailed description of $\Sigma^m(F_n)$ requires a lot of terminology and notation: we show that for $2\le n,m$, if $\chi=a\chi_0+c_0\psi_0+\cdots+c_{n-3}\psi_{n-3}+b\chi_1$ is a character of $F_n$ then $[\chi]$ fails to lie in $\Sigma^m(F_n)$ if and only if all $c_i$ are zero, and both $a$ and $b$ are non-negative. The reader will have to consult Section~\ref{sec:groups_and_chars} to see what all this means.

Computing $\Sigma$-invariants has historically proved difficult, and here one difficulty is in finding a way to realize an arbitrary character of $F_n$ as a height function on $X_n$. We do this by first introducing some measurements (``proto-characters'') on $n$-ary trees and forests, and extrapolating these to characters on $F_n$ and height functions on $X_n$. Once all the characters are cataloged, we use Morse theory and combinatorial arguments to compute all the $\Sigma^m(F_n)$. One key tool, Lemma~\ref{lem:popular_simplex}, is a new technique for proving higher connectivity properties of a simplicial complex, building off of recent work of Belk and Forrest.

A pleasant consequence of Theorem~A is the following, which is immediate from Citation~\ref{cit:bnsr_fin_props} below, plus the aforementioned fact that every proper quotient of $F_n$ is abelian.

\begin{corollary*}
 Let $N$ be any normal subgroup of $F_n$. Then as soon as $N$ is finitely presented, it is already of type $F_\infty$. \qed
\end{corollary*}

It should be noted that it is possible to find subgroups of $F_n$ that are finitely presented but not of type $\FP_3$, and hence not of type $\F_\infty$ \cite[Theorem~B]{bieri10}. However, the corollary says that for normal subgroups this cannot happen.

Another immediate application of Theorem~A comes from \cite{kochloukova14}, namely Kochloukova's Theorem~C in that paper holds for all $F_n$. In words, not only is the deficiency gradient of $F_n$ zero with respect to any chain of finite index subgroups with index going to infinity, but so too are all the higher dimensional analogs. This can be viewed as a strong finiteness property. For more details and background, see \cite{kochloukova14}.

\medskip

At the end of the present work, we discuss the problem of computing the higher $\Sigma$-invariants of the \emph{Houghton groups} $H_n$. The group $H_n$ is of type $\F_{n-1}$ but not $\F_n$ \cite[Theorem~5.1]{brown87}, so one can ask what $\Sigma^m(H_n)$ is for $1\le m\le n-1$. We compute large parts of each $\Sigma^m(H_n)$ (Theorem~\ref{thrm:houghton_pos}), using the action of $H_n$ on a $\CAT(0)$ cube complex, and conjecture that anything not accounted for by the theorem must lie outside $\Sigma^m(H_n)$ (Conjecture~\ref{conj:houghton_neg}). The conjecture holds for $m=1,2$, but it seems that proving it for higher $m$ will require new ideas.

\medskip

The paper is organized as follows. After some topological setup in Section~\ref{sec:prelims}, we define the groups $F_n$ and their characters in Section~\ref{sec:groups_and_chars}. In Section~\ref{sec:stein_farley} we discuss a $\CAT(0)$ cube complex $X_n$ on which $F_n$ acts, and in Section~\ref{sec:links_matchings} we provide a combinatorial model for links in $X_n$. In Section~\ref{sec:computations} we prove Theorem~A. Section~\ref{sec:houghton} is devoted to the Houghton groups $H_n$; we compute lower bounds on $\Sigma^m(H_n)$, and discuss the problem of trying to make this bound sharp.

\subsection*{Acknowledgments} I would first like to acknowledge Stefan Witzel, my coauthor on \cite{witzel15}; some of the tools used here (e.g., Lemma~\ref{lem:morse}) were developed there, and working on that paper spurred me to attempt this problem. I am grateful to Robert Bieri and Desi Kochloukova for first suggesting I try this problem and for helpful conversations along the way, and to Matt Brin for many fruitful discussions as well.


\section{Topological setup}\label{sec:prelims}

Let $G$ be a finitely generated group. A \emph{character} of $G$ is a homomorphism $\chi\colon G\to \R$. If $\chi(G)\cong\Z$, then $\chi$ is \emph{discrete}. The \emph{character sphere} of $G$, denoted $S(G)$, is $\Hom(G,\R) \cong \R^d$ with $0$ removed and modulo positive scaling, so $S(G)\cong S^{d-1}$, where $d$ is the rank of $G/[G,G]$. The \emph{Bieri--Neumann--Strebel (BNS) invariant} $\Sigma^1(G)$ of $G$ is the subset of $S(G)$ defined by:
\[
\Sigma^1(G) \defeq \{[\chi]\in S(G)\mid \Gamma_{0\le\chi} \text{ is connected}\} \text{.}
\]
Here $\Gamma$ is the Cayley graph of $G$ with respect to some finite generating set, and $\Gamma_{0\le\chi}$ is the full subgraph spanned by those vertices $g$ with $0\le\chi(g)$. We write $[\chi]$ for the equivalence class of $\chi$ in $S(G)$.

The \emph{Bieri--Neumann--Strebel--Renz (BNSR) invariants}, also called \emph{$\Sigma$-invariants} $\Sigma^m(G)$ ($m\in\N\cup\{\infty\}$), introduced in \cite{bieri88}, are defined for groups $G$ of type $\F_m$. Our working definition for $\Sigma^m(G)$ is almost identical to Definition~8.1 in \cite{bux04}:

\begin{definition}[$\Sigma$-invariants]\label{def:bnsr}
 Let $G$ be of type $\F_m$, and let $Y$ be an $(m-1)$-connected $G$-CW complex. Suppose $Y^{(m)}$ is $G$-cocompact and the stabilizer of any $k$-cell is of type $\F_{m-k}$. For $0\ne\chi\in\Hom(G,\R)$, there is a \emph{character height function}, denoted $h_\chi$, i.e., a continuous map $h_\chi\colon Y\to\R$, such that $h_\chi(gy)=\chi(g)+h_\chi(y)$ for all $y\in Y$ and $g\in G$. Then $[\chi]\in\Sigma^m(G)$ if and only if the filtration $(Y^{t\le h_\chi})_{t\in\R}$ is essentially $(m-1)$-connected\footnote{Meaning that for all $t\in\R$ there exists $s\le t$ such that the inclusion $Y^{t\le h_\chi} \to Y^{s\le h_\chi}$ induces the trivial map in $\pi_k$ for all $k\le m-1$.}.
\end{definition}

Here $Y^{t\le h_\chi}$ is defined to be the full\footnote{A subcomplex is \emph{full} if as soon as it contains a simplex's vertices, it also contains the simplex.} subcomplex of $Y$ supported on those vertices $y$ with $t\le h_\chi(y)$. The only difference between our definition and \cite[Definition~8.1]{bux04} is that we use $Y^{t\le h_\chi}$ instead of $h_\chi^{-1}([t,\infty))$. However, the first filtration is essentially $(m-1)$-connected if and only if the second is, so our definition is equivalent.

As mentioned in \cite{bux04}, this definition of $\Sigma^m(G)$ is independent of the choices of $Y$ and $h_\chi$. We will sometimes abuse notation and write $\chi$ instead of $h_\chi$, for both the character and the character height function.

One important application of the $\Sigma$-invariants is:
\begin{cit}\cite[Theorem~1.1]{bieri10}\label{cit:bnsr_fin_props}
 Let $G$ be a group of type $\F_m$ and $N$ a subgroup of $G$ containing $[G,G]$ (so $N$ is normal). Then $N$ is of type $\F_m$ if and only if for every $\chi\in\Hom(G,\R)$ with $\chi(N)=0$ we have $[\chi]\in\Sigma^m(G)$.
\end{cit}

For example, if $\chi \colon G\twoheadrightarrow\Z$ is a discrete character, then $\ker(\chi)$ is of type $\F_m$ if and only if $[\pm\chi]\in\Sigma^m(G)$.

\medskip

The setup of Definition~\ref{def:bnsr} is particularly tractable in the situation where $Y$ is an affine cell complex and $\chi$ is affine on cells. Then discrete Morse theory enters the picture, and higher (essential) connectivity properties can be deduced from higher connectivity properties of ascending/descending links.

An \emph{affine cell complex} $Y$ is the quotient of a disjoint union of euclidean polytopes modulo an equivalence relation that maps every polytope injectively into $Y$, with images called \emph{cells}, such that such cells intersect in faces (see \cite[Definition~I.7.37]{bridson99}). In particular, every cell has an affine structure. The link $\lk_Y v$ of a vertex $v$ of $Y$ is the set of directions in $Y$ emanating out of $v$. The link is naturally a spherical simplicial complex, whose closed cells consist of directions pointing into closed cells of $Y$. If every cell is a cube of some dimension, we call $Y$ an affine cube complex.

The following is taken directly from \cite{witzel15}:

\begin{definition}[Morse function]\label{def:morse}
 The most general kind of \emph{Morse function} on $Y$ that we will be using is a map $(h,s) \colon Y \to \R \times \R$ such that both $h$ and $s$ are affine on cells. The codomain is ordered lexicographically, and the conditions for $(h,s)$ to be a Morse function are the following: the function $s$ takes only finitely many values on vertices of $Y$, and there is an $\varepsilon > 0$ such that every pair of adjacent vertices $v$ and $w$ either satisfy $\abs{h(v) - h(w)} \ge \varepsilon$, or else $h(v) = h(w)$ and $s(v)\ne s(w)$.
\end{definition}

Let us summarize some setup from \cite{witzel15}: We call $h$ the \emph{height}, $s$ the \emph{secondary height} and $(h,s)$ the \emph{refined height}. Every cell has a unique vertex of maximal refined height and a unique vertex of minimal refined height. The \emph{ascending star} $\st^{(h,s)\uparrow}_Y v$ of a vertex $v$ (with respect to $(h,s)$) is the subcomplex of $\st_Y v$ consisting of cells $\sigma$ such that $v$ is the vertex of minimal refined height in $\sigma$. The \emph{ascending link} $\lk^{(h,s)\uparrow}_Y v$ of $v$ is the link of $v$ in $\st^{(h,s)\uparrow}_Y v$. The \emph{descending star} and the \emph{descending link} are defined analogously. Since $h$ and $s$ are affine, ascending and descending links are full subcomplexes. We denote by $Y^{p \le h \le q}$ the full subcomplex of $Y$ supported on vertices $v$ with $p \le h(v) \le q$.

With our definition of Morse function as above, we have the following Morse Lemma, which was proved in \cite{witzel15} (compare to \cite[Corollary~2.6]{bestvina97}):

\begin{lemma}[Morse Lemma]\label{lem:morse}
 Let $p,q,r \in \R \cup \{\pm \infty\}$ with $p \le q \le r$. If for every vertex $v \in Y^{q < h \le r}$ the descending link $\lk^{(h,s)\downarrow}_{Y^{p \le h}} v$ is $(k-1)$-connected then the pair $(Y^{p \le h \le r},Y^{p \le h \le q})$ is $k$-connected. If for every vertex $v \in Y^{p \le h < q}$ the ascending link $\lk^{(h,s)\uparrow}_{Y^{h \le r}} v$ is $(k-1)$-connected then the pair $(Y^{p \le h \le r},Y^{q \le h \le r})$ is $k$-connected.
\end{lemma}

\begin{proof}
 For the sake of keeping things self-contained, we redo the proof from \cite{witzel15}.

 The ``ascending'' version is like the ``descending'' version with $(h,s)$ replaced by $-(h,s)$, so we only prove the descending version. Using induction (and compactness of spheres if $r = \infty$) we can assume that $r-q \le \varepsilon$, where $\varepsilon > 0$ is as in Definition~\ref{def:morse}. By compactness of spheres, it suffices to show that there exists a well order $\preceq$ on the vertices of $Y^{q < h \le r}$ such that the pair
 \[
 (S_{\preceq v},S_{\prec v}) \defeq \left(Y^{p \le h \le q} \cup \bigcup_{w \preceq v} \st^{(h,s)\downarrow}_{Y^{p \le h}} w \text{, } Y^{p \le h \le q} \cup \bigcup_{w \prec v} \st^{(h,s)\downarrow}_{Y^{p \le h}} w\right)
 \]
 is $k$-connected for every vertex $v\in Y^{q < h \le r}$. Let $\preceq$ be any well order satisfying $v\prec v'$ whenever $s(v)<s(v')$ (this exists since $s$ takes finitely many values on vertices). Note that $S_{\preceq v}$ is obtained from $S_{\prec v}$ by coning off $S_{\prec v} \cap \partial \st v$. We claim that this intersection equals the boundary $B$ of $\st^{(h,s)\downarrow} v$ in $Y_{p \le h}^{(h,s)\le (h,s)(v)}$, which is homeomorphic to $\lk^{(h,s)\downarrow}_{Y^{p \le h}} v$ and hence $(k-1)$-connected by assumption. The inclusion $S_{\prec v} \cap \partial \st v \subseteq B$ is evident. Since $S_{\prec v} \cap \partial \st v$ is a full subcomplex of $\partial \st v$, for the converse it suffices to verify that any vertex $w$ adjacent to $v$ with $(h,s)(w) < (h,s)(v)$ lies in $S_{\prec v}$. If $h(w) < h(v)$ then $h(w) \le h(v) - \varepsilon \le r - \varepsilon \le q$, so $w \in Y^{p \le h \le q}$. Otherwise $s(w) < s(v)$ and hence $w \prec v$.
\end{proof}

In practice, the following form is all we will need.

\begin{corollary}\label{cor:morse}
 If $Y$ is $(m-1)$-connected and for every vertex $v \in Y^{h < q}$ the ascending link $\lk^{(h,s)\uparrow}_Y v$ is $(m-1)$-connected, then $Y^{q \le h}$ is $(m-1)$-connected.
\end{corollary}

\begin{proof}
 This follows from the Morse Lemma using $p=-\infty$ and $r=\infty$.
\end{proof}


\section{The groups and characters}\label{sec:groups_and_chars}

Thompson's group $F$ admits many generalizations. In this paper we will be concerned with a family of groups usually denoted $F_{n,\infty}$, which we abbreviate to $F_n$ ($2\le n\in\N$); the group $F_2$ is $F$. As a warning, when dealing with generalizations of Thompson groups, e.g., in \cite{brown87,brin98}, the notation $F_n$ often refers to a different group, in which $F_{n,\infty}$ sits with finite index (not to mention that $F_n$ also often denotes the free group of rank $n$). We will not be concerned with these though, so here \textbf{the notation $F_n$ will always refer to the group denoted $F_{n,\infty}$ in \cite{brown87,brin98,kochloukova12}}. In this section we give three viewpoints of $F_n$ and its characters. The three viewpoints of $F_n$ are: its standard infinite presentation, as a group of homeomorphisms of $[0,1]$, and as a group of $n$-ary tree pairs. The equivalence of these was proved in the original paper by Brown \cite[Section~4]{brown87}. For all three ways of viewing $F_n$, we also discuss characters of $F_n$ from that viewpoint. The last one will be the most important, since it is the one we use later to compute the $\Sigma^m(F_n)$.

\subsection{Presentation}\label{sec:presentation}

The standard infinite presentation for $F_n$ (\cite[Proposition~4.8]{brown87}) is
$$F_n\cong\gen{x_i~(i\in\N_0)\mid x_j x_i = x_i x_{j+(n-1)} \text{ for all }i<j}\text{.}$$
It is easy to abelianize this presentation, and get that $F_n/[F_n,F_n] \cong \Z^n$. One basis for this is $\bar{x}_0,\dots,\bar{x}_{n-1}$. From this, one could get a basis for $\Hom(F_n,\R) \cong \R^n$ by taking the dual basis. This was one tool used in \cite{kochloukova12} to compute $\Sigma^2(F_n)$.

\subsection{Piecewise linear homeomorphisms}\label{sec:homeos}

A more hands-on basis for $\Hom(F_n,\R)$ can be described by viewing $F_n$ as piecewise linear self homeomorphisms of $[0,1]$. We will not prove anything in this subsection, since the model for $F_n$ we will actually use comes in the next subsection; here we are just giving some intuition for $F_n$ and its characters. Each element $f\in F_n$ is an orientation preserving homeomorphism $f\colon[0,1] \to [0,1]$ that is piecewise linear with slopes powers of $n$, and whose finitely many points of non-differentiability lie in $\Z[1/n]$. Already this gives us two interesting characters, usually denoted $\chi_0$ and $\chi_1$. The character $\chi_0$ is the log base $n$ of the right derivative at $0$, and $\chi_1$ is the log base $n$ of the left derivative at $1$.

Any such $f\in F_n$ is determined by certain sets of \emph{breakpoints} in the domain and range, as we now describe. Build a finite set $P \subseteq [0,1]$ by starting with the points $\{0,1\}$, and then do finitely many iterations of the following procedure:
\begin{quote}
 Pick two points $x$ and $x'$ already in $P$, with no points in between them yet in $P$, and then add to $P$ the $n-1$ new points $\frac{(n-i)x+ix'}{n}$ for $0<i<n$.
\end{quote}
For example, after one iteration of this, $P$ consists of $\{0,1/n,2/n,\dots,(n-1)/n,1\}$. Call $P$ a \emph{legal set of breakpoints}. If $Q$ is another legal set of breakpoints with $|P|=|Q|$, then we can define $f \colon [0,1] \to [0,1]$ by sending the points of $P$, in order, to the points of $Q$, and then extending affinely between breakpoints. By construction, slopes will be powers of $n$ and breakpoints will lie in $\Z[1/n]$. Moreover, every $f\in F_n$ arises in this way \cite[Proposition~4.4]{brown87}.

One can show that every element of $\Z[1/n] \cap [0,1]$ appears in some legal set of breakpoints. Moreover, while a point can appear in more than one legal set of breakpoints, and have a different ``position'' in different legal sets of breakpoints, the ``position modulo $n-1$'' is a well defined measurement. The equivalence classes induced by this measurement are in fact the $F_n$-orbits in $\Z[1/n] \cap (0,1)$. (Again, proofs are left to the reader.) For each $0\le i\le n-2$, let $O_i$ denote the $F_n$-orbit of points of $\Z[1/n] \cap (0,1)$ appearing in a legal set of breakpoints in a position congruent to $i$ modulo $n-1$.

Now we can define characters on $F_n$. For a point $x\in(0,1]$ define $LD|_x \colon F_n \to \Z$ to be the log base $n$ of the left derivative at $x$. Similarly for $x\in[0,1)$ let $RD|_x$ be the log base $n$ of the right derivative at $x$. These are not group homomorphisms. However, summing these over a complete $F_n$-orbit $O_i$ would define a homomorphism. To get these sums to be finite, we will actually sum up $LD|_x - RD|_x$, since then for a given $f$ this can be nonzero at only finitely many points. For $0\le i\le n-2$ define:
$$\psi_i(f) \defeq \sum\limits_{x\in O_i} LD|_x (f) - RD|_x (f) \text{.}$$
This is a group homomorphism $\psi_i \colon F_n \to \Z$. As a remark, the characters $-\chi_0$ and $\chi_1$ are also of this form, namely for $-\chi_0$ we sum over the orbit of $0$ (which is just $\{0\}$) and for $\chi_1$ we sum over the orbit $\{1\}$. (Technically this only makes sense if we declare $LD|_0=0$ and $RD|_1=0$.)

Note that $\sum_{i=0}^{n-2} \psi_i = \chi_0 - \chi_1$. However, one can check that $\chi_0,\psi_0,\dots,\psi_{n-3},\chi_1$ are linearly independent, and so form a basis of $\Hom(F_n,\R)\cong \R^n$. In the next subsection we will redefine the $\psi_i$ using a different model for $F_n$, and in particular will prove all of these facts.

\subsection{$n$-ary trees}\label{sec:trees}

This brings us to the descriptions of the $F_n$ and their characters that we will use for the rest of the paper, namely making use of $n$-ary trees.

An \emph{$n$-ary tree} will always mean a finite connected tree with a single vertex of degree $n$ or $0$, its \emph{root}, some number of degree $1$ vertices, the \emph{leaves}, and all other vertices of degree $n+1$. The \emph{trivial tree} $\elnothing$ is the one where the root has degree $0$ (so there are no leaves or other vertices). The \emph{$n$-caret} $\elsplit_n$ is the non-trivial $n$-ary tree in which every vertex is either the root or a leaf. Every $n$-ary tree can be obtained as a union of $n$-carets.

For an $n$-ary tree $T$, each leaf of $T$ has a unique reduced path to the root. The length of this path (i.e., its number of edges) defines the \emph{depth} of that leaf. As a remark, the trivial tree is characterized as having a leaf of depth $0$, and the $n$-caret is characterized as having all its leaves of depth $1$.

For each $n$-ary tree $T$, say with $r$ leaves, we fix a planar embedding of $T$, and hence an order on the leaves. We label the leaves $0$ through $r-1$, left to right. The next definition is of various measurements that we will call \emph{proto-characters} on $T$, which will later be used to define characters on elements of $F_n$.

\begin{definition}[Proto-characters]\label{def:proto_chars}
 Let $T$ be an $n$-ary tree with leaves labeled $0$ through $r-1$, left to right. Define $L(T)$ to be the depth of the $0$th leaf. Define $R(T)$ to be the depth of the $(r-1)$st leaf. For each $0\le j\le r-1$, define $d_j(T)$ to be the depth of the $j$th leaf, and then for each $0\le j\le r-2$ define
 $$\delta_j(T) \defeq d_j(T) - d_{j+1}(T) \text{.}$$
 This is the \emph{$j$th change of depth} of $T$. For each $0\le i\le n-2$ define
 $$D_i(T) \defeq \sum \{\delta_j(T) \mid 0\le j \le r-2 \text{, } j\equiv i \mod (n-1)\} \text{.}$$
\end{definition}

As a quick (and trivial) example, if $\elsplit_n$ is the $n$-caret then $L(\elsplit_n)=R(\elsplit_n)=1$, and $D_i(\elsplit_n)=0$ for all $i$, since $d_j(\elsplit_n)=1$ for all $j$. A less trivial example is given in Figure~\ref{fig:proto}.

\begin{figure}[htb]
 \begin{tikzpicture}\centering
  \draw (0,0) -- (1,1) -- (2,0)   (1,1) -- (1,0)   (-1,-1) -- (0,0) -- (1,-1)   (0,0) -- (0,-1)   (-1,-2) -- (0,-1) -- (1,-2)   (0,-1) -- (0,-2);
  \filldraw (-1,-1) circle (1.5pt)   (0,-2) circle (1.5pt)   (1,-1) circle (1.5pt)   (2,0) circle (1.5pt);
  \filldraw[white] (-1,-2) circle (1.5pt)   (1,-2) circle (1.5pt)   (1,0) circle (1.5pt);
  \draw (-1,-2) circle (1.5pt)   (1,-2) circle (1.5pt)   (1,0) circle (1.5pt);
 \end{tikzpicture}
 \caption{A $3$-ary tree $T$ with $r=7$. Leaves $0$, $2$, $4$ and $6$ are labeled with a black dot (those congruent to $0$ mod $2$), and leaves $1$, $3$ and $5$ with a white dot (congruent to $1$ mod $2$). Visibly, $L(T)=2$ and $R(T)=1$. To compute $D_0$, we add $\delta_0+\delta_2+\delta_4$ and get $D_0(T)=-1+0+1=0$. To compute $D_1$ we add $\delta_1+\delta_3+\delta_5$ and get $D_1(T)=0+1+0=1$.}\label{fig:proto}
\end{figure}
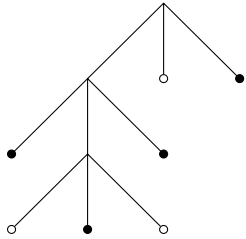

An \emph{$n$-ary tree pair} $(T_-,T_+)$ consists of $n$-ary trees $T_-$ and $T_+$ such that $T_-$ and $T_+$ have the same number of leaves. Two $n$-ary tree pairs are \emph{equivalent} if they can be transformed into each other via a sequence of reductions and expansions. An \emph{expansion} amounts to adding an $n$-caret to the $k$th leaf of $T_-$ and one to the $k$th leaf of $T_+$, for some $k$. A \emph{reduction} is the reverse of an expansion. We denote the equivalence class of $(T_-,T_+)$ by $[T_-,T_+]$.

These $[T_-,T_+]$ are the elements of $F_n$. The multiplication, say of $[T_-,T_+]$ and $[U_-,U_+]$, written $[T_-,T_+] \cdot [U_-,U_+]$, is defined as follows. First note that $T_+$ and $U_-$ admit an $n$-ary tree $S$ that contains them both, so using expansions we have $[T_-,T_+] = [\hat{T}_-,S]$ and $[U_-,U_+] = [S,\hat{U}_+]$ for some $\hat{T}_-$ and $\hat{U}_+$. Then we define
$$[T_-,T_+] \cdot [U_-,U_+] \defeq [\hat{T}_-,S] \cdot [S,\hat{U}_+] = [\hat{T}_-,\hat{U}_+] \text{.}$$
This multiplication is well defined, and it turns out the resulting structure is a group, namely $F_n$.

Having described elements of $F_n$ using the $n$-ary tree pair model, we now describe characters. We make use of the proto-characters from Definition~\ref{def:proto_chars}.

\begin{definition}[Characters]\label{def:chars}
 Let $f=[T,U]=(T,U)\in F_n$. Define
 $$\chi_0(f) \defeq L(U) - L(T) \text{ and } \chi_1(f) \defeq R(U) - R(T) \text{.}$$
 For $0\le i\le n-2$ define
 $$\psi_i(f) \defeq D_i(U) - D_i(T) \text{.}$$
\end{definition}

\begin{lemma}\label{lem:chars_well_def}
 The functions $\chi_0$, $\chi_1$ and $\psi_i$ ($0\le i\le n-2$) are well defined group homomorphisms from $F_n$ to $\Z$.
\end{lemma}

\begin{proof}
 For well definedness, we need to show that for $\chi\in\{\chi_0,\chi_1,\psi_i\}_{i=0}^{n-2}$, if $T'$ (respectively $U'$) is obtained from $T$ (respectively $U$) by adding an $n$-caret to the $k$th leaf, then $\chi(T',U')=\chi(T,U)$. If suffices to show that for $A\in \{L,R,D_i\}_{i=0}^{n-2}$, the value $A(T')-A(T)$ depends only on $i$, $k$ and $r$, where $r$ is the number of leaves of $T$. Since $U$ has the same number of leaves, this will show that $A(T')-A(T) = A(U')-A(U)$, and so $A(U')-A(T') = A(U)-A(T)$ and $\chi(T',U')=\chi(T,U)$. For $A=L,R$ this is clear: $L(T')-L(T)=1$ if $k=0$ and $L(T')-L(T)=0$ otherwise, and $R(T')-R(T)=1$ if $k=r-1$ and $R(T')-R(T)=0$ otherwise. Now let $A=D_i$. We then have the following:
 \begin{enumerate}
  \item If $0<k$ and $k-1\equiv_{n-1} i$, then $D_i(T')=D_i(T)-1$.
  \item If $k<r-1$ and $k\equiv_{n-1} i$, then $D_i(T')=D_i(T)+1$.
  \item Otherwise $D_i(T')-D_i(T) = 0$.
 \end{enumerate}
 In particular, $D_i(T')-D_i(T)$ depends only on $i$, $k$ and $r$.

 It is now easy to check that the $\chi$ are group homomorphisms. If we have two elements to multiply, represent them with a common tree and get $[T,U]\cdot [U,V] = [T,V]$; then for $A\in\{L,R,D_i\}$ we have $A(U) - A(T) + A(V) - A(U) = A(V) - A(T)$, so any $\chi\in\{\chi_0,\chi_1,\psi_i\}$ is a homomorphism.
\end{proof}

As the proof showed, we now know how the measurements $L$, $R$ and $D_i$ change when an $n$-caret is added to the $k$th leaf of an $n$-ary tree. For example if $0<k$ and $k-1\equiv_{n-1} i$ then $D_i$ goes down by $-1$, and if $k<r-1$ and $k\equiv_{n-1} i$ then $D_i$ goes up by $1$. See Figure~\ref{fig:proto_change} for an example.

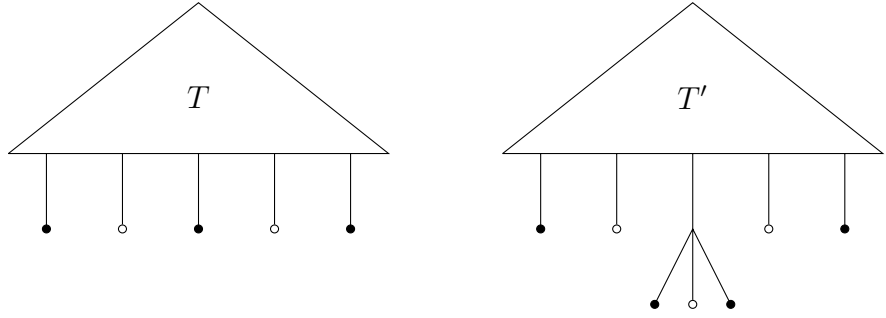
\begin{figure}[htb]
 \begin{tikzpicture}\centering
  \draw (-0.5,0) -- (4.5,0) -- (2,2) -- (-0.5,0)   (0,0) -- (0,-1)   (1,0) -- (1,-1)   (2,0) -- (2,-1)   (3,0) -- (3,-1)   (4,0) -- (4,-1);
  \filldraw (0,-1) circle (1.5pt)   (2,-1) circle (1.5pt)   (4,-1) circle (1.5pt);
  \filldraw[white] (1,-1) circle (1.5pt)   (3,-1) circle (1.5pt);
  \draw (1,-1) circle (1.5pt)   (3,-1) circle (1.5pt);
  \node at (2,0.75) {$T$};

  \begin{scope}[xshift=6.5cm]
   \draw (-0.5,0) -- (4.5,0) -- (2,2) -- (-0.5,0)   (0,0) -- (0,-1)   (1,0) -- (1,-1)   (2,0) -- (2,-1)   (3,0) -- (3,-1)   (4,0) -- (4,-1);
   \filldraw (0,-1) circle (1.5pt)   (4,-1) circle (1.5pt);
   \filldraw[white] (1,-1) circle (1.5pt)   (3,-1) circle (1.5pt);
   \draw (1,-1) circle (1.5pt)   (3,-1) circle (1.5pt);
   \draw (1.5,-2) -- (2,-1) -- (2.5,-2)   (2,-1) -- (2,-2);
   \filldraw (1.5,-2) circle (1.5pt)   (2.5,-2) circle (1.5pt);
   \filldraw[white] (2,-2) circle (1.5pt);
   \draw (2,-2) circle (1.5pt);
   \node at (2,0.75) {$T'$};
  \end{scope}
 \end{tikzpicture}
 \caption{A $3$-caret is added to the second leaf (counting starts at zero) of some $3$-ary tree $T$ with five leaves, to get a new $3$-ary tree $T'$, so $k=2$ and $r=5$. For $T$ we have some changes of depth $\delta_0,\dots,\delta_3$, and for $T'$ we have some changes of depth $\delta_0',\dots,\delta_5'$. The relationships are $\delta_0'=\delta_0$, $\delta_1'=\delta_1-1$, $\delta_2'=0$, $\delta_3'=0$, $\delta_4'=\delta_2+1$ and $\delta_5'=\delta_3$. Hence $D_0'=D_0+1$ and $D_1'=D_1-1$.}\label{fig:proto_change}
\end{figure}

\begin{proposition}[Basis]\label{prop:char_basis}
 As elements of $\Hom(F_n,\R)\cong \R^n$, the $n$ characters
 $$\chi_0,\psi_0,\dots,\psi_{n-3},\chi_1$$
 are linearly independent, and hence form a basis. A dependence involving $\psi_{n-2}$ is that $\psi_0+\cdots+\psi_{n-2}=\chi_0 - \chi_1$.
\end{proposition}

\begin{proof}
 For the second statement, just note that for any tree $T$, $D_0(T)+\cdots D_{n-2}(T) = L(T) - R(T)$.

 We turn to the statement about linear independence. For $0\le k\le n-1$, let $T_k$ be the tree consisting of an $n$-caret with another $n$-caret on its $k$th leaf, so $T_k$ has leaves labeled $0$ through $2n-2$. It is straightforward to compute $L(T_0)=2$, $L(T_k)=1$ for $k>0$, $R(T_{n-1})=2$, $R(T_k)=1$ for $k<n-1$, and the following for the $D_i$ ($0\le i\le n-2$):
 \begin{align*}
  D_i(T_k) = \left\{\begin{array}{ll} -1 & \text{if } i = k-1 \\
                                       1 & \text{if } i = k \\
                                       0 & \text{else.}
                    \end{array}\right.
 \end{align*}
 For $0\le i\le n-2$, we therefore have
 $$(D_i(T_0),D_i(T_1),\dots,D_i(T_{n-1})) = (0,\dots,0,1,-1,0,\dots,0)$$
 with the $1$ at $D_i(T_i)$. We will also need to use trees $T_k'$, obtained by attaching the root of $T_k$ to the last leaf of an $n$-caret. For each $k$ we have $L(T_k')=1$, $R(T_k')=R(T_k)+1$ and $D_i(T_k')=D_i(T_k)$ for all $0\le i\le n-3$.

 Consider the $n$ elements $[T_0,T_{n-1}],\dots,[T_{n-2},T_{n-1}],[T_0',T_{n-1}']$ of $F_n$. Our goal now is to hit them with the $n$ characters $\chi_0,\psi_0,\dots,\psi_{n-3},\chi_1$ to get an $n$-by-$n$ matrix, and then show that this matrix is non-singular. In particular this will prove that these $n$ characters are linearly independent. The $\chi_0$ row is $(-1,0,\dots,0)$ and the $\chi_1$ row is $(1,\dots,1)$. For $0\le i\le n-3$, $D_i(T_{n-1})=0$, so $\psi_i([T_k,T_{n-1}])=-D_i(T_k)$ for $0\le k\le n-2$, and similarly $\psi_i([T_0',T_{n-1}'])=-D_i(T_0')$. Hence we can compute the rows for $\psi_i$, using our previous computation of the $D_i(T_k)$. We get that the $\psi_0$ row is $(-1,1,0,\dots,0,-1)$, the $\psi_1$ row is $(0,-1,1,0,\dots,0)$, and so forth up to the $\psi_{n-3}$ row, which is $(0,\dots,0,-1,1,0)$. Arranging these rows into a matrix, we need to show non-singularity of the matrix:
 \begin{align*}
  \begin{pmatrix}
   -1 &  0 &  0 &  0 & \dots &  0 & 0 & 0 \\
   -1 &  1 &  0 &  0 & \dots &  0 & 0 & -1\\
    0 & -1 &  1 &  0 & \dots &  0 & 0 & 0 \\
    0 &  0 & -1 &  1 & \dots &  0 & 0 & 0 \\
    0 &  0 &  0 & -1 & \dots &  0 & 0 & 0 \\
   \vdots & \vdots & \vdots & \vdots & \ddots & \vdots & \vdots & \vdots \\
    0 &  0 &  0 &  0 & \dots &  1 & 0 & 0 \\
    0 &  0 &  0 &  0 & \dots & -1 & 1 & 0 \\
    1 &  1 &  1 &  1 & \dots &  1 & 1 & 1
  \end{pmatrix}
 \end{align*}
 This is visibly ``almost'' lower triangular; the second row (the $\psi_0$ row) is the only problem, if $n>2$ (note that if $n=2$ then the only rows are $\chi_0$ and $\chi_1$, and this matrix is lower triangular and non-singular). We hit this row with elementary row operations, namely if $r_i$ is the $i$th row we replace $r_2$ with
 $$r_2 + r_n - r_{n-1} - 2r_{n-2} - 3r_{n-3} - \cdots - (n-3)r_3 \text{.}$$
 The new second row is $(0,n-1,0,\dots,0)$, and hence the matrix reduces to a lower triangular matrix whose determinant is readily computed to be $-(n-1)$. This matrix is therefore non-singular, and so the characters $\chi_0,\psi_0,\dots,\psi_{n-3},\chi_1$ are linearly independent elements of $\Hom(F_n,\R)\cong \R^n$.
\end{proof}

Clearly this proof would have been faster if, instead of $\psi_0$, we used the character $\psi_0 + \chi_1 - \psi_{n-3} - 2\psi_{n-4} - 3\psi_{n-5} - \cdots - (n-3)\psi_1$, but since computing the $\Sigma^m(F_n)$ will involve being able to tell whether our basis characters increase, decrease, or neither under certain moves, it will be advantageous to have basis characters with the easiest possible descriptions.

\begin{remark}
 The $\psi_i$ here agree with the $\psi_i$ in Subsection~\ref{sec:homeos}, provided the connection between the homeomorphism model and the $n$-ary tree pair model is made correctly. For $(T,U)$ we view $U$ as the ``domain tree'' and $T$ as the ``range tree''. Each tree defines a subdivision of $[0,1]$ into as many subintervals as there are leaves. Then, the subdivision given by the domain tree is taken to the subdivision given by the range tree, defining a homeomorphism as described in Subsection~\ref{sec:homeos}. It is straightforward to check that the two definitions of $\psi_i$ agree.
\end{remark}


\section{Stein--Farley complexes}\label{sec:stein_farley}

In this section we recall the Stein--Farley $\CAT(0)$ cube complex $X_n$ on which $F_n$ acts, and extend the characters $\chi \colon F_n \to \R$ to functions $\chi \colon X_n \to \R$. The complex $X_n$ was first constructed by Stein \cite{stein92} building off ideas of Brown, and shown to be $\CAT(0)$ by Farley \cite{farley03}, who viewed $F_n$ as a \emph{diagram group}, \`a la Guba and Sapir \cite{guba97}. To define $X_n$, we first expand from considering $n$-ary trees to considering $n$-ary forests. An \emph{$n$-ary forest} is a disjoint union of finitely many $n$-ary trees. The roots and leaves of the trees are \emph{roots} and \emph{leaves} of the forest. We fix an order on the trees, and hence on the leaves. An \emph{$n$-ary forest pair} $(E_-,E_+)$ consists of $n$-ary forests $E_-$ and $E_+$ such that $E_-$ and $E_+$ have the same number of leaves. We call the roots of $E_-$ \emph{heads} and the roots of $E_+$ \emph{feet} of the pair (the terminology comes from flipping $E_+$ upside down and identifying the leaves of $E_-$ and $E_+$).

Just like the tree case, we have a notion of equivalence. Two $n$-ary forest pairs are \emph{equivalent} if they can be transformed into each other via a sequence of reductions or expansions. We denote the equivalence class of $(E_-,E_+)$ by $[E_-,E_+]$. Let $\calP$ be the set of equivalence classes of $n$-ary forest pairs.

This set has two important pieces of structure. First, it is a groupoid. If $[E_-,E_+]$ has $k$ heads and $\ell$ feet, and $[D_-,D_+]$ has $\ell$ heads and $m$ feet, then we can define their product, written $[E_-,E_+] \cdot [D_-,D_+]$, which is an $n$-ary forest pair with $k$ heads and $m$ feet. Like in $F_n$, with $n$-ary tree pairs, to define the product we first note that $E_+$ and $D_-$ admit an $n$-ary forest $C$ that contains them both. Then applying expansions we can write $[E_-,E_+] = [\hat{E}_-,C]$ and $[D_-,D_+] = [C,\hat{D}_+]$ for some $\hat{E}_-$ and $\hat{D}_+$, and then define
$$[E_-,E_+] \cdot [D_-,D_+] \defeq [\hat{E}_-,C] \cdot [C,\hat{D}_+] = [\hat{E}_-,\hat{D}_+] \text{.}$$
For $\calP$ to be a groupoid with this multiplication, we need identities and inverses. A forest in which all trees are trivial is called a \emph{trivial forest}. The trivial forest with $\ell$ trees is denoted $\id_\ell$. We can view an $n$-ary forest $E$ as an $n$-ary forest pair via $E \mapsto [E,\id_\ell]$, where $\ell$ is the number of leaves of $E$. It is clear that for any element with $k$ heads and $\ell$ feet, $[\id_k,\id_k]$ is the left identity and $[\id_\ell,\id_\ell]$ is the right identity. We also have inverses, namely the (left and right) inverse of $[E_-,E_+]$ is $[E_+,E_-]$.

Since $F_n$ lives in $\calP$ as the set of elements with one head and one foot, we have an action of $F_n$, by multiplication, on the subset $\calP_1$ of elements with one head.

The second piece of structure on $\calP$ is an order relation. The order is defined by: $[E_-,E_+] \le [D_-,D_+]$ whenever there is an $n$-ary forest $C$ such that $[E_-,E_+] \cdot C = [D_-,D_+]$. We informally refer to right multiplication by an $n$-ary forest pair of the form $[C,\id_\ell]$ as \emph{splitting} the feet of $[E_-,E_+]$. Multiplying by $[\id_\ell,C]$ is called \emph{merging}. This terminology comes from viewing $E_+$ upside down with its leaves attached to those of $E_-$, forming a ``strand diagram'' \`a la \cite{belk14}. It is straightforward to check that $\le$ is a partial order, so $\calP$ is a poset. The subset $\calP_1$ of elements with one head is a subposet.

The topological realization of the poset $(\calP_1,\le)$ is a contractible simplicial complex on which $F_n$ acts, and the \emph{Stein--Farley complex} $X_n$ is a certain invariant subcomplex with a natural cubical structure. Given $n$-ary forest pairs $[E_-,E_+] \le [E_-,E_+] \cdot E$, write $[E_-,E_+] \preceq [E_-,E_+] \cdot E$ whenever $E$ is an \emph{elementary $n$-ary forest}. This means that each $n$-ary tree of $E$ is either trivial or a single $n$-caret. Now $X_n$ is defined to be the subcomplex of $|\calP_1|$ consisting of chains $x_0<\cdots<x_k$ with $x_i \preceq x_j$ for all $i\le j$. The cubical structure is given by intervals: given $x\preceq y$ ($x$ with $r$ feet), the interval $[x,y]\defeq\{z\mid x\le z\le y\}$ is a Boolean lattice of dimension $r$, and so the simplices in $[x,y]$ form an $r$-cube. Note that $x\prec y$ are adjacent, i.e., share a $1$-cube, if and only if $y=x \cdot E$ for $E$ an elementary $n$-ary forest with just a single $n$-caret.

\begin{theorem}\cite{farley03}
 $X_n$ is a $\CAT(0)$ cube complex.
\end{theorem}

Every cube $\sigma$ has a unique vertex $x$ with fewest feet and a unique vertex $y$ with most feet. There is a unique elementary $n$-ary forest $E$ with $y=x\cdot E$, and the other vertices of $\sigma$ are obtained by multiplying $x$ by subforests of $E$. We use the following notation: suppose $x$ has $\ell$ feet and $E=(A_0,\dots,A_{\ell-1})$, where each $A_i$ is either $\elnothing$ or $\elsplit_n$; here $\elnothing$ is the trivial tree and $\elsplit_n$ is the tree with one $n$-caret. Let $\Phi$ be the set of subforests of $E$, written $\Phi \defeq \gen{A_0,\ldots,A_{\ell-1}}$. Then the vertex set of $\sigma$ is precisely $x\Phi$.

If we center ourselves at a different vertex $z$ of $\sigma$, then we also have to allow merges. Say $z$ has $r > \ell$ feet. Then we can write $\sigma = z\Psi$ where $\Psi$ is of the form $\gen{A_0,\ldots,A_{\ell-1}}$, where each $A_i$ is either $\elnothing$, $\elsplit_n$ or $\elmerge_n$. Here $\elmerge_n$ is the inverse of the tree with one $n$-caret (so an upside-down $n$-caret). The tuple $(A_0,\ldots,A_{\ell-1})$ can be thought of as an $n$-ary forest pair, with all the $\elsplit_n$ in the first forest and all the $\elmerge_n$ in the second forest (and some $\elnothing$s included if necessary). Then the set $\Psi$ is the set of all $n$-ary forest pairs that can be obtained by removing some of the carets. As before, the vertex set of $\sigma$ is $z\Psi$.

Note that the action of $F_n$ on $X_n$ is free, since the action on vertices is given by multiplication in a groupoid, and if an element stabilizes a cube $[x,y]$ then it fixes $x$.

\subsection{Character height functions}\label{sec:char_height_fxns}

In Subsection~\ref{sec:trees}, we defined characters on $F_n$ by first defining ``proto-characters'' on $n$-ary trees, and then viewing elements of $F_n$ as $n$-ary tree pairs. It is straightforward to extend these proto-characters to be defined on $n$-ary forests. To be precise, each leaf of an $n$-ary forest is connected to a unique root, which gives it a depth, so $n$-ary forests $E$ admit the measurements $L(E)$, $R(E)$, $\delta_j(E)$ and $D_i(E)$. Much like the proto-characters on $n$-ary trees induce the characters (group homomorphisms) $\chi_0$, $\chi_1$ and $\psi_i$ from $F_n$ to $\Z$, also the proto-characters on $n$-ary forests induce groupoid homomorphisms $\calP \to \Z$ extending these characters.

In particular, the $\chi_i$ and $\psi_i$ can now be evaluated on vertices of $X_n$. Moreover, any character $\chi$ on $F_n$ can be written as a linear combination
\begin{equation}
\label{eq:character_linear_combination}
\chi = a\chi_0 + c_0\psi_0 + \cdots + c_{n-3}\psi_{n-3} + b\chi_1
\end{equation}
thanks to Proposition~\ref{prop:char_basis}. Hence $\chi$ extends to arbitrary $n$-ary forest pairs by interpreting \eqref{eq:character_linear_combination} as a linear combination of the extended characters.

It will be important to know how our basis characters vary between adjacent vertices of $X_n$.

\begin{lemma}[Varying characters]\label{lem:vary_chars}
 Let $x$ be a vertex in $X_n$, say with feet numbered $0$ through $r-1$, left to right. Let $\elsplit_n(r,k)$ be the elementary $n$-ary forest with $r$ roots and a single non-trivial tree, namely an $n$-caret on the $k$th root. Let $y=x\cdot \elsplit_n(r,k)$. We have the following:
 \begin{enumerate}
  \item If $k=0$ then $\chi_0(y)=\chi_0(x)-1$.
  \item If $k>0$ then $\chi_0(y)=\chi_0(x)$.
  \item If $k<r-1$ then $\chi_1(y)=\chi_1(x)$.
  \item If $k=r-1$ then $\chi_1(y)=\chi_1(x)-1$.
  \item If $0<k$ and $k-1\equiv_{n-1} i$, then $\psi_i(y) = \psi_i(x) + 1$.
  \item If $k<r-1$ and $k\equiv_{n-1} i$, then $\psi_i(y) = \psi_i(x) - 1$.
  \item Otherwise $\psi_i(y) = \psi_i(x)$.
 \end{enumerate}
\end{lemma}

\begin{proof}
 Let $\chi\in\{\chi_0,\chi_1,\psi_i\}_{i=0}^{n-3}$ be a basis character. Let $A\in\{L,R,D_i\}_{i=0}^{n-3}$ be the corresponding proto-character. Since $\chi$ is a groupoid morphism $\calP \to \Z$, we have
 $$\chi(y) = \chi(x)+\chi([\elsplit_n(r,k),\id_{r+(n-1)}]) = \chi(x) - A(\elsplit_n(r,k)) \text{.}$$
 Hence, to check the cases in the statement, it suffices to check the following, all of which are readily verified:
 \begin{enumerate}
  \item If $k=0$ then $L(\elsplit_n(r,k)) = 1$.
  \item If $k>0$ then $L(\elsplit_n(r,k)) = 0$.
  \item If $k<r-1$ then $R(\elsplit_n(r,k)) = 0$.
  \item If $k=r-1$ then $R(\elsplit_n(r,k)) = 1$.
  \item If $0<k$ and $k-1\equiv_{n-1} i$, then $D_i(\elsplit_n(r,k)) = -1$.
  \item If $k<r-1$ and $k\equiv_{n-1} i$, then $D_i(\elsplit_n(r,k)) = 1$.
  \item Otherwise $D_i(\elsplit(r,k)) = 0$.
 \end{enumerate}
\end{proof}

Note that since we only consider $0\le i\le n-3$, if $y=x \cdot \elsplit_n(r,r-1)$ (i.e., if we get from $x$ to $y$ by splitting the last foot), then no $\psi_i$ changes, since $r-2\equiv_{n-1} n-2$.

\medskip

So far we know that any character $\chi$ on $F_n$ can be extended to all the vertices of $X_n$. Now we extend it to the entire complex.

\begin{lemma}\label{lem:affine_extend}
 Any character $\chi$ extends to an affine map $\chi \colon X_n \to \R$.
\end{lemma}

Before proving this, we reduce the problem using the following:

\begin{lemma}\cite[Lemma~2.4]{witzel15}\label{lem:affine_on_cubes}
 Let $\varphi \colon \{0,1\}^r \to \R$ be a map that can be affinely extended to the $2$-faces of the cube $[0,1]^r$. Then $\varphi$ can be affinely extended to all of $[0,1]^r$.
\end{lemma}

\begin{proof}
 This was proved in \cite{witzel15}, and we repeat the proof here for the sake of being self contained. There is a unique affine function $\tilde{\varphi} \colon \R^r \to\R$ that agrees with $\varphi$ on the zero vector and the $r$ standard basis vectors. We claim that $\tilde{\varphi}$ agrees with $\varphi$ on all the other vertices of $[0,1]^r$ as well, and hence defines an affine extension of $\varphi$ to all of $[0,1]^r$. Let $v=(v_1,\dots,v_r)$ be a vertex with at least two entries equal to $1$ (and the others all $0$). Pick $i\ne j$ with $v_i=v_j=1$. For any $w$ obtained from $v$ by zeroing out $v_i$, $v_j$, or both, we have by induction that $\tilde{\varphi}(w)=\varphi(w)$. These three $w$ vertices, plus $v$, define a $2$-face of $[0,1]^r$. By assumption, $\varphi$ can be affinely extended to this $2$-face, and the value on $v$ is uniquely determined by the values on the other three vertices. Hence $\tilde{\varphi}(v)=\varphi(v)$.
\end{proof}

\begin{proof}[Proof of Lemma~\ref{lem:affine_extend}]
 Let $\square_2=v \Phi$ be a $2$-cube in $X_n$, say $\Phi=\gen{A_0,\dots,A_{r-1}}$, with exactly two $A_i$ being $\elsplit_n$ and all others being $\elnothing$. Thanks to Lemma~\ref{lem:affine_extend}, we just need to show that $\chi$ extends affinely to $\square_2$. Say $A_j=A_k=\elsplit_n$ for $j<k$, and let $v_j=v\gen{\elnothing,\dots,A_j,\dots,\elnothing}$, $v_k=v\gen{\elnothing,\dots,A_k,\dots,\elnothing}$ and $v_{j,k}=v\gen{\elnothing,\dots,A_j,\dots,A_k,\dots,\elnothing}$. Hence the vertices of $\square_2$ are $v$, $v_j$, $v_k$ and $v_{j,k}$. Now we just need to show that $\chi(v_j)-\chi(v) = \chi(v_{j,k})-\chi(v_k)$.

 It suffices to do this for $\chi\in\{\chi_0,\chi_1,\psi_i\}_{i=0}^{n-3}$. It is clear that $\chi_0(v_j)-\chi_0(v) = \chi_0(v_{j,k})-\chi_0(v_k)$, namely they equal $-1$ if $j=0$ and equal $0$ otherwise, and similarly we always have $\chi_1(v_j)-\chi_1(v) = \chi_1(v_{j,k})-\chi_1(v_k) = 0$. Next consider $\psi_i$. By Lemma~\ref{lem:vary_chars}, we have that $\psi_i(v_j)-\psi_i(v)=1$ if and only if $0<j$ and $j-1\equiv_{n-1} i$, which also holds if and only if $\psi_i(v_{j,k})-\psi_i(v_k)=1$. Also, $\psi_i(v_j)-\psi_i(v)=-1$ if and only if $j\equiv_{n-1} i$ if and only if $\psi_i(v_{j,k})-\psi_i(v_k)=-1$ (since $j<k$, we know $j$ cannot be the highest index of a foot of either $v$ or $v_k$). The only other option is $\psi_i(v_j)-\psi_i(v) = \psi_i(v_{j,k})-\psi_i(v_k) = 0$.
\end{proof}

These extended characters $\chi$ will be our height functions. Our secondary height will be given by the number of feet function $f$.

\begin{observation}\label{obs:feet_affine}
 There is a map $f \colon X_n \to \R$ that is affine on cubes and assigns to any vertex its number of feet. It is a Morse function.
\end{observation}

\begin{proof}
 That $f$ extends affinely is straightforward. When we say that $f$ is a Morse function, in the language of Definition~\ref{def:morse} this means that $(f,0)$ is a Morse function. This is true because adjacent vertices $v$ and $w$ satisfy $\abs{f(v)-f(w)}=n-1$.
\end{proof}

Let $X_n^{p\le f\le q}$ be the subcomplex of $X_n$ supported on vertices $v$ with $p\le f(v)\le q$.

\begin{proposition}\label{prop:char_morse}
 Let $\chi$ be a character. The pair $(\chi,f)$ is a Morse function on $X_n^{p\le f\le q}$ for any $p\le q<\infty$.
\end{proposition}

\begin{proof}
 We check the conditions required by Definition~\ref{def:morse}. We have extended $\chi$ and $f$ to affine functions in Lemma~\ref{lem:affine_extend} and Observation~\ref{obs:feet_affine}. By construction $f$ takes finitely many values on $X_n^{p\le f\le q}$. Write $\chi=a\chi_0+c_0\psi_0+\cdots+c_{n-3}\psi_{n-3}+b\chi_1$. Let
 $$\varepsilon\defeq \min\{\abs{d}\mid d=\alpha a + \beta b + \gamma_0 c_0 +\cdots+\gamma_{n-3} c_{n-3} \ne 0 \text{ for } \alpha,\beta,\gamma_i\in\{-1,0,1\}\}\text{.}$$
 Since we only consider such $d$ that are non-zero, and there are finitely many, we have $0<\varepsilon$. For any pair of adjacent vertices $v$ and $w$, we know from Lemma~\ref{lem:vary_chars} that for any basis character $\phi\in\{\chi_0,\chi_1,\psi_i\}_{i=0}^{n-3}$, we have $\phi(v) - \phi(w) \in \{-1,0,1\}$. Hence for any character $\chi$ we have $\chi(v) - \chi(w) = \alpha a + \beta b + \gamma_0 c_0 +\cdots+\gamma_{n-3} c_{n-3}$ for some $\alpha,\beta,\gamma_i\in\{-1,0,1\}$. In particular, either $\abs{\chi(v) - \chi(w)} \ge \varepsilon$ or else $\chi(v)=\chi(w)$. The condition $f(v)\ne f(w)$ is always satisfied anyway for adjacent vertices, so we conclude that $(\chi,f)$ is a Morse function.
\end{proof}


\section{Links and matchings}\label{sec:links_matchings}

We will use Morse theory to reduce the computation of $\Sigma^m(F_n)$ to questions about ascending links in $X_n$. In this section we discuss a useful model for links in $X_n$.

\begin{definition}
 Let $\Delta$ be a simplicial complex, say of dimension $d$. Let $D\subseteq\{0,\dots,d\}$. A \emph{$D$-matching} is a subset $\mu$ of $\Delta$, consisting of $k$-simplices for $k\in D$ such that any two simplices in $\mu$ are disjoint. If $D=\{k\}$ is a singleton we may write ``$k$-matching'' instead of ``$\{k\}$-matching''. For example a $0$-matching is just any collection of $0$-simplices, and a $1$-matching is what is usually called a \emph{matching} on the graph $\Delta^{(1)}$. For our purposes, we will be interested in certain $(n-1)$-dimensional complexes $\Delta=\Delta^n(r)$, defined below, and $D=\{0,n-1\}$, so $D$-matchings are collections of pairwise disjoint $0$-simplices and $(n-1)$-simplices. In general, the $D$-matchings of $\Delta$ form a simplicial complex, denoted $\match_D(\Delta)$, with face relation given by inclusion, called the \emph{$D$-matching complex} of $\Delta$.
\end{definition}

Define $\Delta^n(r)$ as follows. It is a simplicial complex on $r$ vertices, labeled $v_0$ through $v_{r-1}$, such that a collection of vertices spans a simplex precisely when $\abs{i-j}<n$ for all vertices $v_i$ and $v_j$ in the collection. For example, $\Delta^1(r)$ is a discrete set of $r$ vertices, and $\Delta^2(r)$ is the linear graph on $r$ vertices. The complex $\Delta^3(9)$ is shown in Figure~\ref{fig:Delta^3(9)}. To keep notation straight, we reiterate that $r$ is the number of vertices of $\Delta^n(r)$, and $n$ is the maximum number of vertices that may share a simplex.

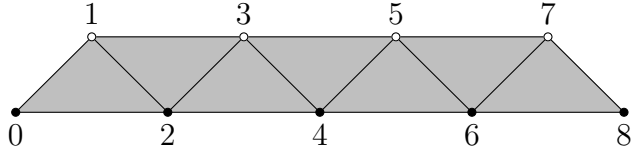
\begin{figure}[htb]
 \begin{tikzpicture}\centering
  \filldraw[lightgray] (0,0) -- (8,0) -- (7,1) -- (1,1);
  \draw (0,0) -- (1,1) -- (2,0) -- (3,1) -- (4,0) -- (5,1) -- (6,0) -- (7,1) -- (8,0)   (0,0) -- (8,0)   (1,1) -- (7,1);
  \filldraw (0,0) circle (1.5pt)   (2,0) circle (1.5pt)   (4,0) circle (1.5pt)   (6,0) circle (1.5pt)   (8,0) circle (1.5pt);
  \filldraw[white] (1,1) circle (1.5pt)   (3,1) circle (1.5pt)   (5,1) circle (1.5pt)   (7,1) circle (1.5pt);
  \draw (1,1) circle (1.5pt)   (3,1) circle (1.5pt)   (5,1) circle (1.5pt)   (7,1) circle (1.5pt);
  \node at (0,-.3) {$0$};   \node at (1,1.3) {$1$};   \node at (2,-.3) {$2$};   \node at (3,1.3) {$3$};    \node at (4,-.3) {$4$};   \node at (5,1.3) {$5$};   \node at (6,-.3) {$6$};   \node at (7,1.3) {$7$};   \node at (8,-.3) {$8$};
 \end{tikzpicture}
 \caption{The complex $\Delta^3(9)$. The vertices are numbered $0$ to $8$, left to right, with the even vertices labeled by a black circle and the odd vertices labeled by a white circle. (The distinction will be important later.)}\label{fig:Delta^3(9)}
\end{figure}

For any $0\le i\le j\le r-1$ with $j-i<n$, let $e_{[i,j]}$ denote the $(j-i)$-simplex $\{i,i+1,\dots,j\}$, so $\{e_{[i,j]}\}$ is a $0$-simplex in the $(j-i)$-matching complex. When a matching $\{e\}$ consists of a single simplex $e$, we will usually abuse notation and just write $e$ for the matching. For example $e_{[i,j]}$ now represents both a $(j-i)$ simplex in $\Delta^n(r)$ and a $0$-simplex in $\match_{j-i}(\Delta^n(r))$, and $v_k$ represents both a $0$-simplex in $\Delta^n(r)$ and a $0$-simplex in $\match_0(\Delta^n(r))$.

\begin{lemma}[$(n-1)$-matchings]\label{lem:top_matching_conn}
 For $n,r\in\N$, $\match_{n-1}(\Delta^n(r))$ is $(\lfloor\frac{r-n}{2n-1}\rfloor-1)$-connected.
\end{lemma}

\begin{proof}
 Note that $n$ is fixed. We induct on $r$. The base case is that $\match_{n-1}(\Delta^n(r))$ is non-empty when $n\le r$, which is true. Now assume that $3n-1\le r$. In this case, for any $(n-1)$-matching $\mu$ in $\match_{n-1}(\Delta^n(r))$, either $e_{[i,i+(n-1)]}\in \mu$ for some $0\le i\le n-1$, or else every $0$-simplex of $\mu$ is an $(n-1)$-simplex of $\Delta^n(r)$ that is disjoint from $e_{[0,n-1]}$. In particular, $\match_{n-1}(\Delta^n(r))$ is covered by the contractible subcomplexes $S_i\defeq \st(e_{[i,i+(n-1)]})$ for $0\le i\le n-1$. The $S_i$ all contain the matching $e_{[r-n,r-1]}$, since $3n-1\le r$ implies $2n-2<r-n$, so the nerve of the covering is contractible (a simplex). Any intersection $S_{i_1}\cap\cdots \cap S_{i_t}$ for $t>1$ is isomorphic to a matching complex of the form $\match_{n-1}(\Delta^n(r'))$ for $r'\ge r-(2n-1)$. By induction this is $(\lfloor\frac{r-(2n-1)-n}{2n-1}\rfloor-1)$-connected, and hence $(\lfloor\frac{r-n}{2n-1}\rfloor-2)$-connected. The result now follows from the Nerve Lemma \cite[Lemma~1.2]{bjoerner94}.
\end{proof}

For example, $\match_2(\Delta^3(9))$ is connected, which is clear from Figure~\ref{fig:Delta^3(9)}.

\medskip

There is an analogy between $\{0,n-1\}$-matchings on $\Delta^n(r)$ and points in the link of a vertex $x\in X_n$ with $r$ feet. That is, each $0$-matching is a single vertex of $\Delta^n(r)$, so corresponds to splitting a foot of $x$ into $n$ new feet, and each $(n-1)$-matching is a collection of $n$ sequential vertices of $\Delta^n(r)$, so corresponds to merging $n$ sequential feet of $x$ into one new foot. We make this rigorous in the next lemma.

Let $x$ be a vertex of $X_n$ with $r$ feet. The cofaces of $x$ are the cells $\sigma = x\Psi$, for every $\Psi$ such that $x\Psi$ makes sense. If $\Psi=\gen{A_0,\dots,A_{\ell-1}}$ for $A_i\in\{\elnothing,\elsplit_n,\elmerge_n\}$ ($0\le i\le \ell-1$), then the rule is that $r$ must equal the number of $A_i$ that are $\elnothing_n$ or $\elsplit_n$, plus $n$ times the number that are $\elmerge_n$.

\begin{lemma}[Link model]\label{lem:vertex_link}
 If a vertex $x\in X_n$ has $r$ feet then $\lk x \cong \match_{\{0,n-1\}}(\Delta^n(r))$.
\end{lemma}

\begin{proof}
 Define a map $g\colon \lk x \to \match_{\{0,n-1\}}(\Delta^n(r))$ as follows. For a coface $x\Psi$ with $\Psi=\gen{A_0,\dots,A_{\ell-1}}$, $g$ sends $x\Psi$ to a $\{0,n-1\}$-matching of $\Delta^n(r)$ where each $\elsplit_n$ is a $0$-simplex in $\Delta^n(r)$ and each $\elmerge_n$ is an $(n-1)$-simplex in $\Delta^n(r)$. More precisely, for each $0\le i\le \ell-1$, let $m_i$ be the number of $0\le j<i$ such that $A_j=\elmerge_n$, and then
 \[
 g(x\Psi) \defeq \{v_{k+(n-1)m_k},e_{[\ell+(n-1)m_\ell,\ell+(n-1)m_\ell+(n-1)]}\mid A_k=\elsplit_n \text{ and } A_\ell=\elmerge_n\} \text{.}
 \]
 For example, $g(x\gen{\elnothing,\elmerge_n,\elsplit_n})=\{e_{[1,n]},v_{n+1}\}$. It is straightforward to check that $g$ is a simplicial isomorphism.
\end{proof}

See Figure~\ref{fig:lk_model} for an example of the correspondence $\lk x \cong \match_{\{0,n-1\}}(\Delta^n(r))$.

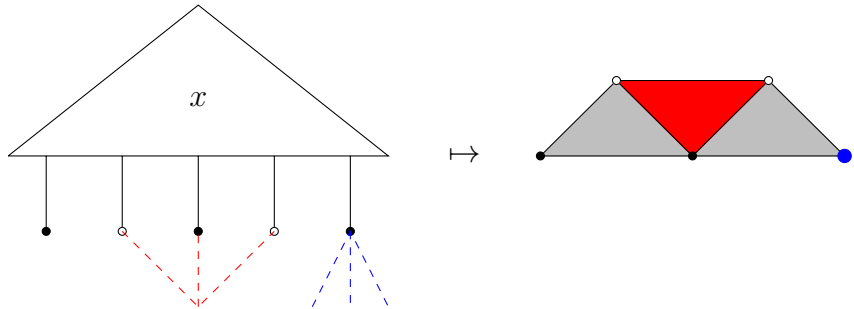
\begin{figure}[htb]
 \begin{tikzpicture}\centering
  \draw (-0.5,0) -- (4.5,0) -- (2,2) -- (-0.5,0)   (0,0) -- (0,-1)   (1,0) -- (1,-1)   (2,0) -- (2,-1)   (3,0) -- (3,-1)   (4,0) -- (4,-1);
  \filldraw (0,-1) circle (1.5pt)   (2,-1) circle (1.5pt)   (4,-1) circle (1.5pt);
  \filldraw[white] (1,-1) circle (1.5pt)   (3,-1) circle (1.5pt);
  \draw (1,-1) circle (1.5pt)   (3,-1) circle (1.5pt);
  \draw[red,dashed] (1,-1) -- (2,-2) -- (3,-1)   (2,-2) -- (2,-1);
  \draw[blue,dashed] (3.5,-2) -- (4,-1) -- (4.5,-2)   (4,-1) -- (4,-2);
  \node at (2,0.75) {$x$};
  \node at (5.5,0) {$\mapsto$};
  
  \begin{scope}[xshift=6.5cm]
   \filldraw[lightgray] (0,0) -- (4,0) -- (3,1) -- (1,1);
   \filldraw[red] (1,1) -- (2,0) -- (3,1);
   \draw (0,0) -- (1,1) -- (2,0) -- (3,1) -- (4,0)   (0,0) -- (4,0)   (1,1) -- (3,1);
   \filldraw (0,0) circle (1.5pt)   (2,0) circle (1.5pt);
   \filldraw[blue] (4,0) circle (2.5pt);
   \filldraw[white] (1,1) circle (1.5pt)   (3,1) circle (1.5pt);
   \draw (1,1) circle (1.5pt)   (3,1) circle (1.5pt);
  \end{scope}

 \end{tikzpicture}
 \caption{The example $x\gen{\elnothing,\elmerge_3,\elsplit_3} \mapsto \{e_{[1,3]},v_4\}$ from the proof of Lemma~\ref{lem:vertex_link}. The $\elmerge_3$ and $e_{[1,3]}$ are in red and the $\elsplit_3$ and $v_4$ are in blue.}\label{fig:lk_model}
\end{figure}

For a vertex $x \in X_n$ recall that $f(x)$ denotes its number of feet. The function $f$ extends to an affine Morse function on $X_n$ (Observation~\ref{obs:feet_affine}). Viewing $\lk x$ as $\match_{\{0,n-1\}}(\Delta^n(r))$ for $x$ with $f(x)=r$, the ascending link of $x$ with respect to $f$ is $\match_0(\Delta^n(r))$ and the descending link is $\match_{n-1}(\Delta^n(r))$.

\begin{corollary}[$f$-ascending/descending links]\label{cor:f_lks}
 For $x$ a vertex with $f(x)=r$, $\lk^{f\uparrow}_{X_n} x$ is contractible and $\lk^{f\downarrow}_{X_n} x$ is $(\lfloor\frac{r-n}{2n-1}\rfloor-1)$-connected.
\end{corollary}

\begin{proof}
 We have $\lk^{f\uparrow}_{X_n} x \cong \match_0(\Delta^n(r))$, which is an $(r-1)$-simplex, hence contractible. We have $\lk^{f\downarrow}_{X_n} x \cong \match_{n-1}(\Delta^n(r))$, which is $(\lfloor\frac{r-n}{2n-1}\rfloor-1)$-connected by Lemma~\ref{lem:top_matching_conn}.
\end{proof}

In Section~\ref{sec:computations} we will need a subcomplex of the form $X_n^{p \le f \le q}$ that is $(m-1)$-connected. It will be convenient to have one of the form $X_n^{p \le f \le pn^2}$.

\begin{lemma}\label{lem:sublevel_conn}
 For any $p\ge m$ the complex $X_n^{p \le f \le pn^2}$ is $(m-1)$-connected.
\end{lemma}

\begin{proof}
 We first claim that $X_n^{f \le pn^2}$ is $(m-1)$-connected. By the Morse Lemma (specifically Corollary~\ref{cor:morse}) it suffices to show that for any vertex $x$ with $f(x)>pn^2$, the $f$-descending link $\lk^{f\downarrow}_{X_n} x$ is $(m-1)$-connected. Setting $r=f(x)$, we know from Corollary~\ref{cor:f_lks} that the $f$-descending link is $(\lfloor\frac{r-n}{2n-1}\rfloor-1)$-connected. Since $r\ge pn^2+1 \ge mn^2+1$, this is $(\lfloor\frac{mn^2-n+1}{2n-1}\rfloor-1)$-connected. To see that $mn^2-n+1 \ge m(2n-1)$ (which now suffices), we note that the roots of the polynomial $mx^2+(-2m-1)x+(m+1)$ are $1$ and $1+\frac{1}{m}$.

 Now we pass from $X_n^{1 \le f \le pn^2}$ to $X_n^{p \le f \le pn^2}$. In fact these are homotopy equivalent, since ascending links of vertices with respect to $f$ are contractible (Corollary~\ref{cor:f_lks}), and for a vertex with fewer than $p$ feet, the entire ascending link is contained in $X_n^{1 \le f \le pn^2}$.
\end{proof}

\begin{observation}\label{obs:cocompact}
 For $p,q\in\N$, the action of $F_n$ on $X_n^{p \le f \le q}$ is cocompact.
\end{observation}

\begin{proof}
 For each $r$, $F_n$ acts transitively on vertices with $r$ feet. The result is thus immediate since $X_n$ is locally compact.
\end{proof}

\medskip

In particular, we now have highly connected spaces on which our groups act freely and cocompactly, which is part of the setup for Definition~\ref{def:bnsr}. To compute the $\Sigma$-invariants using Morse theory, we will use our knowledge of how characters vary between adjacent vertices (Lemma~\ref{lem:vary_chars}). Since we are modeling vertex links by $\{0,n-1\}$-matching complexes on $\Delta^n(r)$, we need to translate Lemma~\ref{lem:vary_chars} into the language of $\{0,n-1\}$-matchings.

\begin{definition}
 Let $\chi$ be a character of $F_n$, extended to $X_n$ as in Lemma~\ref{lem:affine_extend}. Let $x\in X_n$ be a vertex with $r=f(x)$ feet, so $\lk x \cong \match_{\{0,n-1\}}(\Delta^n(r))$. Under this isomorphism, call a vertex of $\match_{\{0,n-1\}}(\Delta^n(r))$ \emph{$\chi$-ascending} if the corresponding vertex $y$ in $\lk x$ has $\chi(y)>\chi(x)$. Analogously define \emph{$\chi$-descending} and \emph{$\chi$-preserving}. Say a simplex $\mu$ in $\match_{\{0,n-1\}}(\Delta^n(r))$ is $\chi$-ascending/descending/preserving if all its vertices are.
\end{definition}

\begin{observation}[Ascending matching complexes]\label{obs:asc_lks_matchings}
 Let $(\chi,f) \colon X_n^{p\le f\le q} \to \R\times \R$ be a Morse function as in Proposition~\ref{prop:char_morse}. Let $x$ be a vertex in $X_n^{p\le f\le q}$ with $r=f(x)$ feet. Then the $(\chi,f)$-ascending link of $x$ in $X_n$ is isomorphic to the full subcomplex of $\match_{\{0,n-1\}}(\Delta^n(r))$ supported on those $0$-simplices $v_k$ and $e_{[k,k+(n-1)]}$ such that $v_k$ is either $\chi$-ascending or $\chi$-preserving, and $e_{[k,k+(n-1)]}$ is $\chi$-ascending. The $(\chi,f)$-ascending link of $x$ in $X_n^{p\le f\le q}$ is then obtained by removing any $\{0,n-1\}$-matchings $\mu$ such that $r+(n-1)\mu_0>q$ or $r-(n-1)\mu_{n-1}<p$, where $\mu_i$ is the number of vertices of $\mu$ that are $i$-matchings.
\end{observation}

\begin{proof}
 To increase $(\chi,f)$, we must either increase $\chi$ or else preserve $\chi$ and increase $f$. The $v_k$ correspond to vertices in $\lk x$ with $r+(n-1)$ feet, and the $e_{[k,k+(n-1)]}$ to vertices in $\lk x$ with $r-(n-1)$ feet. Hence the first claim follows. For the second claim, just note that $\mu$ corresponds to a simplex in $\lk x$, and hence to a cube in $X_n$ containing $x$, and $r+(n-1)\mu_0$ is the maximum number of feet of a vertex in that cube; similarly $r-(n-1)\mu_{n-1}$ is the minimum number of feet of a vertex in that cube.
\end{proof}

\begin{corollary}\label{cor:char_match}
 If $k=0$ then $v_k$ is $\chi_0$-descending and $e_{[k,k+(n-1)]}$ is $\chi_0$-ascending. Otherwise they are both $\chi_0$-preserving. If $k=r-1$ then $v_k$ is $\chi_1$-descending and $e_{[k-(n-1),k]}$ is $\chi_1$-ascending. Otherwise they are both $\chi_1$-preserving. If $0<k$ and $k-1\equiv_{n-1} i$, then $v_k$ is $\psi_i$-ascending and $e_{[k,k+(n-1)]}$ is $\psi_i$-descending. If $k<r-1$ and $k\equiv_{n-1} i$, then $v_k$ is $\psi_i$-descending and $e_{[k-(n-1),k]}$ is $\psi_i$-ascending. Anything not covered by these cases is $\psi_i$-preserving. In all of these cases, ``ascending'' entails an increase by $+1$ and ``descending'' entails a decrease by $-1$.
\end{corollary}

\begin{proof}
 Translating to $\lk x$, $v_k$ corresponds to $x \cdot [\elsplit_n(r,k),\id_{r+(n-1)}]$, $e_{[k,k+(n-1)]}$ corresponds to $x \cdot [\id_r,\elsplit_n(r-(n-1),k)]$ and $e_{[k-(n-1),k]}$ corresponds to $x \cdot [\id_r,\elsplit_n(r-(n-1),k-(n-1))]$. Hence Lemma~\ref{lem:vary_chars} implies all of these facts.
\end{proof}

Note that in particular if $r-1>0$ then $v_{r-1}$ is $\psi_i$-preserving for all $0\le i\le n-3$, since $r-2\equiv_{n-1} n-2$.

Some examples of $\psi_i$-ascending, descending or preserving $0$-simplices in $\match_{\{0,2\}}(\Delta_3(5))$, as governed by Corollary~\ref{cor:char_match}, are shown in Figure~\ref{fig:char_match}.

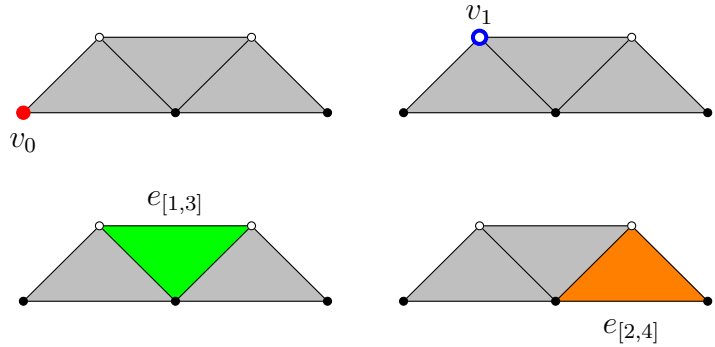
\begin{figure}[htb]
 \begin{tikzpicture}\centering
  \filldraw[lightgray] (0,0) -- (4,0) -- (3,1) -- (1,1);
  \draw (0,0) -- (1,1) -- (2,0) -- (3,1) -- (4,0)   (0,0) -- (2,0) -- (4,0)   (1,1) -- (3,1);
  \filldraw[red] (0,0) circle (2.5pt); \filldraw (2,0) circle (1.5pt)   (4,0) circle (1.5pt);
  \filldraw[white] (1,1) circle (1.5pt)   (3,1) circle (1.5pt);
  \draw (1,1) circle (1.5pt)   (3,1) circle (1.5pt);
  \node at (0,-.4) {$v_0$};
  
  \begin{scope}[xshift=5cm]
   \filldraw[lightgray] (0,0) -- (4,0) -- (3,1) -- (1,1);
   \draw (0,0) -- (1,1) -- (2,0) -- (3,1) -- (4,0)   (0,0) -- (2,0) -- (4,0)   (1,1) -- (3,1);
   \filldraw (0,0) circle (1.5pt)   (2,0) circle (1.5pt)   (4,0) circle (1.5pt);
   \filldraw[white] (1,1) circle (2.5pt)   (3,1) circle (1.5pt);
   \draw[blue,line width=1.5pt] (1,1) circle (2.5pt); \draw (3,1) circle (1.5pt);
   \node at (1,1.3) {$v_1$};
  \end{scope}

  \begin{scope}[yshift=-2.5cm]
   \filldraw[lightgray] (0,0) -- (4,0) -- (3,1) -- (1,1);
   \filldraw[green] (1,1) -- (2,0) -- (3,1);
   \draw (0,0) -- (1,1) -- (2,0) -- (3,1) -- (4,0)   (0,0) -- (2,0) -- (4,0)   (1,1) -- (3,1);
   \filldraw (0,0) circle (1.5pt)   (2,0) circle (1.5pt)   (4,0) circle (1.5pt);
  \filldraw[white] (1,1) circle (1.5pt)   (3,1) circle (1.5pt);
  \draw (1,1) circle (1.5pt)   (3,1) circle (1.5pt);
   \node at (2,1.3) {$e_{[1,3]}$};
  \end{scope}

  \begin{scope}[yshift=-2.5cm,xshift=5cm]
   \filldraw[lightgray] (0,0) -- (4,0) -- (3,1) -- (1,1);
   \filldraw[orange] (2,0) -- (3,1) -- (4,0);
   \draw (0,0) -- (1,1) -- (2,0) -- (3,1) -- (4,0)   (0,0) -- (2,0) -- (4,0)   (1,1) -- (3,1);
   \filldraw (0,0) circle (1.5pt)   (2,0) circle (1.5pt)   (4,0) circle (1.5pt);
  \filldraw[white] (1,1) circle (1.5pt)   (3,1) circle (1.5pt);
  \draw (1,1) circle (1.5pt)   (3,1) circle (1.5pt);
   \node at (3,-.4) {$e_{[2,4]}$};
  \end{scope}

 \end{tikzpicture}
 \caption{Some $0$-simplices in $\match_{\{0,2\}}(\Delta^3(5))$. First is $v_0$ (red), which is $\psi_0$-descending and $\psi_1$-preserving; then $v_1$ (blue) is $\psi_0$-ascending and $\psi_1$-descending; next $e_{[1,3]}$ (green) is $\psi_0$-descending and $\psi_1$-ascending; finally $e_{[2,4]}$ (orange) is $\psi_0$-preserving and $\psi_1$-descending.}\label{fig:char_match}
\end{figure}


\section{Proof of theorem~A}\label{sec:computations}

In this section we prove Theorem~A, that $\Sigma^m(F_n)=\Sigma^2(F_n)$ for all $n,m\ge 2$. The forward inclusion always holds, so the work to do is the reverse inclusion. Throughout this section, $\chi$ is a character of $F_n$ with $[\chi]\in\Sigma^2(F_n)$.

For the first three lemmas, we will make use of a certain ascending HNN-extension of $F_n$. (We should mention that there is nothing novel here, and the reduction done over the course of these three lemmas was already contained in the work of Kochloukova \cite{kochloukova12}.) Let $F_n(1)$ be the subgroup of $F_n$ generated by the $x_i$ for $i>0$ (see Subsection~\ref{sec:presentation}). It is well known that $F_n=F_n(1)*_{x_0}$ and $F_n(1)\cong F_n$.

\begin{lemma}\label{lem:poles}
 If $\chi=-\chi_i$ for $i=0,1$ then $[\chi]\in\Sigma^\infty(F_n)$.
\end{lemma}

\begin{proof}
 By symmetry it suffices to do the $i=0$ case. We know $F_n=F_n(1)*_{x_0}$ and that $F_n(1)\cong F_n$ is of type $\F_\infty$. Also, $-\chi_0(F_n(1))=0$ and $-\chi_0(x_0)=1$, so the result follows from \cite[Theorem~2.1]{bieri10}.
\end{proof}

Now suppose $\chi=a\chi_0+b\chi_1$. Since $[\chi]\in\Sigma^2(F_n)$, we know from \cite[Proposition~9,Theorem~10]{kochloukova12} that $a<0$ or $b<0$. This could also be deduced using the action of $F_n$ on $X_n$, following the proof of the $n=2$ case in \cite{witzel15}. This would take many pages of technical details though, so we content ourselves with just citing Kochloukova to say that we know $a<0$ or $b<0$.

\begin{lemma}\label{lem:equator}
 If $\chi=a\chi_0+b\chi_1$ with $a<0$ or $b<0$ then $[\chi]\in\Sigma^\infty(F_n)$.
\end{lemma}

\begin{proof}
 By symmetry we can assume $b<0$. Since $\chi_0(F_n(1))=0$, we have that $\chi|_{F_n(1)}$ is equivalent to $-\chi_1$ when restricted to $F_n(1)$. Now, $F_n(1)\cong F_n$ is of type $F_\infty$ and $F_n=F_n(1)*_{x_0}$, so by \cite[Theorem~2.3]{bieri10} and Lemma~\ref{lem:poles} $[\chi]\in\Sigma^\infty(F_n)$.
\end{proof}

Now we can assume $n>2$ and $\chi$ has non-zero $\psi_i$ component for some $i$.

\begin{lemma}\label{lem:push_to_hemispheres}
 Assume that we already know every non-trivial character of the form $\chi'=\sum_{i=0}^{n-3} c_i \psi_i$ has $[\chi']\in\Sigma^\infty(F_n)$. Then for any $\chi=a\chi_0 + \sum_{i=0}^{n-3} c_i \psi_i + b\chi_1$ with $c_i\ne 0$ for at least one $i$, we have $[\chi]\in\Sigma^\infty(F_n)$.
\end{lemma}

\begin{proof}
 Note that such a $\chi$ restricted to $F_n(1)$ is still non-trivial. As in the previous proof, we can restrict to $F_n(1)$ and ensure that without loss of generality $a=0$. If $b\ne 0$ then appealing to symmetry, we can rather assume $a\ne 0$ but $b=0$. Now by the first sentence we can reduce to the case $a=b=0$. But this is exactly the case already handled in the assumption.
\end{proof}

This brings us to the final case, where we assume that $\chi$ is a linear combination of the $\psi_i$ for $0\le i\le n-3$, and show that $[\chi]\in\Sigma^\infty(F_n)$. This is where Kochloukova's approach in \cite{kochloukova12} became difficult to extend beyond the $\Sigma^2$ case, and where our setup from the previous sections will prove to be able to handle all the $\Sigma^m$.

Let $m\in\N$ and set $p\defeq 4m+5$. Let $Y\defeq X_n^{p\le f\le pn^2}$. Then $Y$ is $(m-1)$-connected (Lemma~\ref{lem:sublevel_conn}), and $F_n$ acts freely and cocompactly on $Y$ (Observation~\ref{obs:cocompact}), so we have the requisite setup of Definition~\ref{def:bnsr}. (Of course $Y$ would have already been $(m-1)$-connected just using $p=m$, but having $p=4m+5$ will be important in the proof of Proposition~\ref{prop:asc_lk_conn}.) According to Definition~\ref{def:bnsr}, we need to show that $(Y^{t\le\chi})_{t\in\R}$ is essentially $(m-1)$-connected and then we will have $[\chi]\in\Sigma^m(F_n)$. In fact we will show that every $Y^{t\le\chi}$ is $(m-1)$-connected.

The proof that $Y^{t\le\chi}$ is $(m-1)$-connected will be quite technical, so we first sketch the proof here, to serve as an outline for what follows.

\begin{proof}[Sketch of proof]
 Thanks to Morse theory, it suffices to show that all $(\chi,f)$-ascending links of vertices $x$ are $(m-1)$-connected. Since we are working in $Y$, we know the number of feet of $x$ lies between $p$ and $pn^2$. We consider the cases $p\le f(x)\le pn$ and $pn \le f(x) \le pn^2$ separately. In the first case, even if we split every foot of $x$, we remain in $Y$, so all splittings are ``legal''. In particular if there is some ascending splitting move that is joinable in $X_n$ to every other ascending move, then these joins can even take place in $Y$, and the ascending link is a (contractible) cone. It turns out that the move where we split the rightmost foot serves as such a cone point. Now consider the second case, $pn \le f(x) \le pn^2$. Here there may be splitting moves that push us out of $Y$, but every merging move keeps us inside $Y$. It is too much to hope for to find an ascending merging move joinable to every other ascending move. However, we do find an ascending merging move consisting of a ``large'' simplex $\sigma_q$ such that every ascending vertex is joinable to ``almost all'' of $\sigma_q$. We prove in Lemma~\ref{lem:popular_simplex} that this is sufficient to get high connectivity.
\end{proof}

Now we begin the technicalities. First we need a lemma that is a useful tool for proving higher connectivity of certain complexes. Heuristically, if there is a simplex $\sigma$ such that every vertex is joinable to ``most'' of $\sigma$, then we can conclude higher connectivity properties. The case when the complex is finite and flag was proved by Belk and Forrest, and written down by Belk and Matucci in \cite{belk15}. Here we show that the requirement of being finite can easily be relaxed. We also replace the requirement of being flag with something weaker, and rephrase the condition from \cite{belk15} so that in the flag case it is the same. This is a necessary modification, since the complexes we will apply this lemma to in the proof of Proposition~\ref{prop:asc_lk_conn} are not flag.

\begin{definition}
 Let $\Delta$ be a simplicial complex. Two simplices $\rho_1$ and $\rho_2$ are \emph{joinable} to each other if they lie in a common simplex. For a fixed simplex $\sigma$ in $\Delta$, we will call $\Delta$ \emph{flag with respect to $\sigma$} if whenever $\rho$ is a simplex and $\sigma'$ is a face of $\sigma$ such that every vertex of $\rho$ is joinable to every vertex of $\sigma'$, already $\rho$ is joinable to $\sigma'$. For example if $\Delta$ is flag with respect to every simplex, then it is flag.
\end{definition}

\begin{lemma}\label{lem:popular_simplex}
 Let $\Delta$ be a simplicial complex, and let $k\in\N$. Suppose there exists an $\ell$-simplex $\sigma$ such that $\Delta$ is flag with respect to $\sigma$, and for every vertex $v$ in $\Delta$, $v$ is joinable to some $(\ell-k)$-face of $\sigma$. Then $\Delta$ is $(\lfloor\frac{\ell}{k}\rfloor-1)$-connected.
\end{lemma}

\begin{proof}
 First note that our hypotheses on $\Delta$ are preserved under passing to any full subcomplex $\Delta'$ containing $\sigma$. Indeed, joinability of two simplices is preserved under passing to any full subcomplex containing them, so $\Delta'$ is still flag with respect to $\sigma$ and every vertex of $\Delta'$ is still joinable to some $(\ell-k)$-face of $\sigma$. In particular, without loss of generality $\Delta$ is finite. Indeed, any homotopy sphere lies in some full subcomplex $\Delta'$ of $\Delta$ that contains $\sigma$ and is finite, and the complex $\Delta'$ still satisfies our hypotheses since it is full and contains $\sigma$. If the sphere is nullhomotopic in $\Delta'$ then it certainly is nullhomotopic in $\Delta$, so from now on we may indeed assume $\Delta$ is finite.

 We induct on the number $V$ of vertices of $\Delta$. If $V=\ell+1$ then $\Delta=\sigma$ is contractible. Now suppose $V>\ell+1$, so we can choose a vertex $v\in \Delta \setminus \sigma$. The subcomplex obtained from $\Delta$ by removing $v$ along its link $L\defeq \lk v$ has fewer vertices than $\Delta$, is full, and contains $\sigma$, so by the first paragraph and by induction it is $(\lfloor\frac{\ell}{k}\rfloor-1)$-connected. It now suffices to show that $L$ is $(\lfloor\frac{\ell}{k}\rfloor-2)$-connected.
 
 Let $\tau \defeq \sigma\cap \st v$, so $\tau$ also equals $\sigma\cap L$. Since $\Delta$ is flag with respect to $\sigma$, $\tau$ is a face of $\sigma$. Say $\tau$ is an $(\ell-k')$-simplex, which since $v$ is joinable to an $(\ell-k)$-face of $\sigma$ tells us that $k'\le k$. Now let $w$ be a vertex in $L$. By similar reasoning we know that $\tau_w \defeq \sigma \cap \st w$ is an $(\ell-k'_w)$-simplex for $k'_w\le k$. Intersecting the two faces $\tau$ and $\tau_w$ of $\sigma$ thus yields a face $\omega_w$ that is an $(\ell-k'-k'')$-simplex for $k''\le k'_w\le k$. Since $v$ and $w$ are joinable to $\omega_w$, and $\Delta$ is flag with respect to $\sigma$, the edge connecting $v$ and $w$ is also joinable to $\omega_w$. In particular $\omega_w$ is joinable to $w$ in $L$. We have shown that there is an $(\ell-k')$-simplex, $\tau$, in $L$ such that every vertex $w$ of $L$ is joinable in $L$ to an $(\ell-k'-k)$-face of $\tau$, namely any $(\ell-k'-k)$-face of $\omega_w$. We also claim that $L$ is flag with respect to $\tau$. Indeed, if $\rho$ is a simplex in $L$ and $\tau'$ is a face of $\tau$ such that every vertex of $\rho$ is joinable to every vertex of $\tau'$, then $\rho*v$ is a simplex in $\Delta$ all of whose vertices are joinable in $\Delta$ to all the vertices of $\tau'$, so $\rho*v$ is joinable in $\Delta$ to $\tau'$ and indeed $\rho$ is joinable in $L=\lk v$ to $\tau'$. Now we can apply the induction hypothesis to $L$, and conclude that $L$ is $(\lfloor\frac{\ell-k'}{k}\rfloor-1)$-connected, and hence $(\lfloor\frac{\ell}{k}\rfloor-2)$-connected.
\end{proof}

As a trivial example (which works for any $\Delta$), if there exists a simplex $\sigma$ such that every vertex is joinable to some vertex of $\sigma$, so we can use $k=\ell$, then $\Delta$ is $0$-connected. To tie Lemma~\ref{lem:popular_simplex} to the version in \cite{belk15}, note that if $\Delta$ is flag, then $v$ being joinable to an $(\ell-k)$-face of $\sigma$ is equivalent to $v$ being joinable to all but at most $k$ vertices of $\sigma$.

\medskip

We return to the complex $Y=X_n^{p\le f\le pn^2}$ (recall $p=4m+5$) and the problem of showing that every $Y^{t\le\chi}$ is $(m-1)$-connected. Consider
$$h\defeq(\chi,f) \colon Y \to \R \times \R \text{,}$$
ordered lexicographically. This is a Morse function by Proposition~\ref{prop:char_morse}, so by the Morse Lemma~\ref{lem:morse} (specifically Corollary~\ref{cor:morse}), it suffices to show the following:

\begin{proposition}\label{prop:asc_lk_conn}
 Let $x$ be a vertex in $Y$. Then the $h$-ascending link $\lk^{h\uparrow}_Y x$ is $(m-1)$-connected.
\end{proposition}

\begin{proof}
 Let $r\defeq f(x)$. We view $\lk x$ as $\match_{\{0,n-1\}}(\Delta^n(r))$, so the $h$-ascending link is as described in Observation~\ref{obs:asc_lks_matchings}. We will split the problem into two cases: when $r$ is ``small'' and when $r$ is ``big''. First suppose $p\le r\le pn$. In this case, for any $\{0,n-1\}$-matching $\mu$, with $\mu_0$ the number of $0$-simplices in $\mu$ that are $0$-matchings (so $\mu_0 \le pn$), we have $r+(n-1)\mu_0 \le pn + (n-1)pn = pn^2$. In particular the addition of $0$-matchings to a $\{0,n-1\}$-matching will never push us out of $Y$.

 We know from Corollary~\ref{cor:char_match} that $v_{r-1}$ is $\psi_i$-preserving for all $0\le i\le n-3$, and hence $\chi$-preserving. If $e_{[r-n,r-1]}$ represents a vertex of $\lk_Y x$, i.e., if $p\le r-(n-1)$, then $e_{[r-n,r-1]}$ is $\chi$-preserving since $r-1>0$. Hence $v_{r-1}$ is $h$-ascending and $e_{[r-n,r-1]}$ is not, by Observation~\ref{obs:asc_lks_matchings}. But $e_{[r-n,r-1]}$ is the only $0$-simplex of $\match_{\{0,n-1\}}(\Delta^n(r))$ not joinable to $v_{r-1}$, so $\lk^{h\uparrow}_Y x$ is contractible, via the conical contraction $\mu \le \mu\cup\{v_{r-1}\} \ge \{v_{r-1}\}$.

 Now suppose $pn \le r\le pn^2$. In this case, for any $\{0,n-1\}$-matching $\mu$, with $\mu_{n-1}$ the number of $0$-simplices in $\mu$ that are $(n-1)$-matchings (so $n\mu_{n-1} \le r$), we claim that $r-(n-1)\mu_{n-1} \ge p$. Indeed, if $\mu_{n-1} \ge p$, then $r-(n-1)\mu_{n-1} \ge \mu_{n-1} \ge p$, and if $\mu_{n-1}<p$ then $r-(n-1)\mu_{n-1}> pn - (n-1)p =p$. In analogy to the previous case, this means that the addition of $(n-1)$-matchings to a $\{0,n-1\}$-matching will never push us out of $Y$.

 For $0\le q\le n-2$, let $\sigma_q$ be the $((s/2)-1)$-simplex
 $$\sigma_q \defeq \{e_{[q+(n-1),q+2(n-1)]},e_{[q+3(n-1),q+4(n-1)]}, \dots, e_{[q+(s-1)(n-1),q+s(n-1)]}\} \text{,}$$
 where $s\in 2\N$ is as large as possible such that $q+s(n-1)<r-1$; see Figure~\ref{fig:sigma_0} for an example. Since $r\ge pn \ge 9n$, such an $s$ certainly exists. By maximality of $s$, we must have $q+(s+2)(n-1) \ge r-1$, and since $r \ge pn$ and $q\le n-2$, we then have $s\ge \lfloor\frac{(p-3)n+3}{n-1}\rfloor$. By definition $p=4m+5$, and it is straightforward to check that this bound gives us the bound $s\ge 4m+2$.

 We now want to cleverly choose $q$ so that $\sigma_q$ is $\chi$-ascending, and hence $h$-ascending. Recall that $\chi=c_0\psi_0+\cdots+c_{n-3}\psi_{n-3}$, and now also set $c_{n-2}\defeq 0$. Let $0\le q\le n-2$ be any value such that, with subscripts considered mod $(n-1)$, we have $c_{q-1}<c_q$. Since the $c_i$ cannot all be zero, such a $q$ exists. For this choice of $q$, and any $1\le t\le s-1$, we claim that $e_{[q+t(n-1),q+(t+1)(n-1)]}$ is $\chi$-ascending, which will then imply that $\sigma_q \in \lk^{h\uparrow}_Y x$. By Corollary~\ref{cor:char_match}, since $0<q+(n-1)$ and $q+s(n-1)<r-1$, we know that $e_{[q+t(n-1),q+(t+1)(n-1)]}$ is $\psi_{q-1}$-descending (subscript taken mod $(n-1)$), $\psi_q$-ascending, and $\psi_i$-preserving for all other $0\le i\le n-2$. Then since $c_{q-1}<c_q$, Corollary~\ref{cor:char_match} tells us that $e_{[q+t(n-1),q+(t+1)(n-1)]}$ is indeed $\chi$-ascending, and so $h$-ascending, namely it increases $\chi$ by $c_q-c_{q-1}>0$.
 
 With this $h$-ascending $((s/2)-1)$-simplex $\sigma_q$ in hand, we want to apply Lemma~\ref{lem:popular_simplex} to $\lk^{h\uparrow}_Y x$. Note that $\lk^{h\uparrow}_{X_n} x$ is flag, but $\lk^{h\uparrow}_Y x$ might not be, since filling in missing simplices might require pushing $f$ above $pn^2$. However, we claim $\lk^{h\uparrow}_Y x$ is flag with respect to $\sigma_q$. Indeed, if $\rho$ is a simplex and $\sigma'_q$ is a face of $\sigma_q$ such that every vertex of $\rho$ is joinable to every vertex of $\sigma_q'$, then since $X_n$ is flag we can consider the simplex $\rho*\sigma'_q$ in $X_n$, and since $\sigma_q$ consists only of $(n-1)$-matchings, $f$ achieves its maximum on $\rho*\sigma'_q$ already on $\rho$. Hence if $\rho$ came from $Y$, then $\rho*\sigma'_q$ is also in $Y$, and so $\lk^{h\uparrow}_Y x$ is flag with respect to $\sigma_q$. Now we want to show that every vertex of $\lk^{h\uparrow}_Y x$ is joinable to ``most'' of $\sigma_q$. Let $\mu$ be any $0$-simplex in $\match_{\{0,n-1\}}(\Delta^n(r))$. We claim that $\mu$ is joinable to all but at most two vertices of $\sigma_q$. Indeed, if $\mu=\{v_j\}$ then $\mu$ fails to be joinable to $\{e_{[k,k+(n-1)]}\}$ if and only if $k\le j\le k+(n-1)$, and there is at most one such $\{e_{[k,k+(n-1)]}\}$ in $\sigma_q$ with this property. Similarly if $\mu=\{e_{[j,j+(n-1)]}\}$ then $\mu$ fails to be joinable to $\{e_{[k,k+(n-1)]}\}$ if and only if $k\le j\le k+(n-1)$ or $k\le j+(n-1)\le k+(n-1)$, and there are at most two such $\{e_{[k,k+(n-1)]}\}$ in $\sigma_q$ with this property. We now know that for any $0$-simplex $\mu$ in $\lk^{h\uparrow}_Y x$, $\mu$ is joinable in $\lk^{h\uparrow}_Y x$ to an $((s/2)-3)$-face of $\sigma_q$. By Lemma~\ref{lem:popular_simplex}, we conclude that $\lk^{h\uparrow}_Y x$ is $(\lfloor\frac{(s/2)-1}{2}\rfloor-1)$-connected. Recall that $s\ge 4m+2$, and so $\lk^{h\uparrow}_Y x$ is $(m-1)$-connected.
\end{proof}

\begin{figure}[htb]
 \begin{tikzpicture}\centering
  \filldraw[lightgray] (0,0) -- (14,0) -- (13,1) -- (1,1);
  \filldraw[red] (2,0) -- (3,1) -- (4,0)   (6,0) -- (7,1) -- (8,0)   (10,0) -- (11,1) -- (12,0);
  \draw (0,0) -- (14,0) -- (13,1) -- (1,1) -- (0,0)   (1,1) -- (2,0) -- (3,1) -- (4,0) -- (5,1) -- (6,0) -- (7,1) -- (8,0) -- (9,1) -- (10,0) -- (11,1) -- (12,0) -- (13,1);
  \filldraw (0,0) circle (1.5pt)   (2,0) circle (1.5pt)   (4,0) circle (1.5pt)   (6,0) circle (1.5pt)   (8,0) circle (1.5pt)   (10,0) circle (1.5pt)   (12,0) circle (1.5pt)   (14,0) circle (1.5pt);
  \filldraw[white] (1,1) circle (1.5pt)   (3,1) circle (1.5pt)   (5,1) circle (1.5pt)   (7,1) circle (1.5pt)   (9,1) circle (1.5pt)   (11,1) circle (1.5pt)   (13,1) circle (1.5pt);
  \draw (1,1) circle (1.5pt)   (3,1) circle (1.5pt)   (5,1) circle (1.5pt)   (7,1) circle (1.5pt)   (9,1) circle (1.5pt)   (11,1) circle (1.5pt)   (13,1) circle (1.5pt);
  \node at (3,-.3) {$e_{[2,4]}$};   \node at (7,-.3) {$e_{[6,8]}$};   \node at (11,-.3) {$e_{[10,12]}$};
 \end{tikzpicture}
 \caption{The $2$-simplex $\sigma_0$ in $\match_{\{0,2\}}(\Delta^3(15))$. Here $s=6$.}\label{fig:sigma_0}
\end{figure}
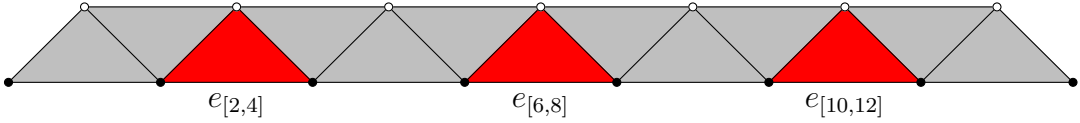

We summarize this section by writing down the proof of Theorem~A.

\begin{proof}[Proof of Theorem~A]
 Let $\chi=a\chi_0+c_0\psi_0+\cdots+c_{n-3}\psi_{n-3}+b\chi_1$ with $[\chi]\in\Sigma^2(F_n)$. If all the $c_i$ are zero then since $[\chi]\in\Sigma^2(F_n)$ we know either $a<0$ or $b<0$, and so $[\chi]\in\Sigma^\infty(F_n)$ by Lemma~\ref{lem:equator}. Now suppose the $c_i$ are not all zero. By Lemma~\ref{lem:push_to_hemispheres}, without loss of generality $a=b=0$. Then by Proposition~\ref{prop:asc_lk_conn} and Corollary~\ref{cor:morse}, $Y^{t\le\chi}$ is $(m-1)$-connected for all $t\in\R$. Hence by Definition~\ref{def:bnsr} we conclude that $[\chi]\in\Sigma^\infty(F_n)$.
\end{proof}


\section{Houghton groups}\label{sec:houghton}

Let $(H_n)_{n\in\N}$ be the family of Houghton groups, introduced in \cite{houghton78}. An element $\eta$ of $H_n$ is an automorphism of $\{1,\dots,n\}\times\N$ such that for each $1\le i\le n$ there exists $m_i\in\Z$ and $N_i\in\N$ such that for all $x\ge N_i$ we have $(i,x)\eta=(i,x+m_i)$. That is, $\eta$ ``eventually acts as translations'' on each ray $\{i\}\times\N$. We have that $H_n$ is of type $\F_{n-1}$ but not of type $\F_n$ \cite[Theorem~5.1]{brown87}.

It is known that $\Hom(H_n,\R)$ is generated by characters $\chi_1,\dots,\chi_n$, given by $\chi_i(\eta)=m_i$ for each $i$ (with $m_i$ as above). Since $\eta$ is an automorphism, $\sum m_i = 0$ for any $\eta$, and hence $\chi_1+\cdots+\chi_n=0$ as characters. In fact for $n\ge 2$, $\chi_1,\dots,\chi_{n-1}$ form a basis of $\Hom(H_n,\R)\cong\R^{n-1}$. Since $H_n$ is of type $\F_{n-1}$, one can ask about $\Sigma^m(H_n)$ for $m\le n-1$. Bieri and Strebel [unpublished], and independently Brown \cite[Proposition~8.3]{brown87bns}, proved that for $n\ge2$ the complement of $\Sigma^1(H_n)$ is $\{[-\chi_i]\}_{i=1}^n$. Note that when $n=2$, $S(H_2)=\{[\chi_1],[-\chi_1]\}$ and $\chi_1=-\chi_2$ so in fact $\Sigma^1(H_2)=\emptyset$.

In this section we prove:

\begin{theorem}\label{thrm:houghton_pos}
 Let $n\in\N$ and let $\chi = a_1\chi_1 + \cdots + a_n \chi_n$ be a non-trivial character, i.e., the $a_i$ are not all equal. Up to symmetry, we can assume $a_1\le \cdots \le a_n$. Let $1\le m(\chi)\le n-1$ be maximal such that $a_{m(\chi)}\ne a_n$. Then $[\chi]\in\Sigma^{m(\chi)-1}(H_n)$.
\end{theorem}

For example, $[\chi_n]\in \Sigma^{n-2}(H_n)$. Note that since $\chi_1+\cdots+\chi_n=0$, without loss of generality $a_{m(\chi)+1}=a_n=0$. With this convention we would for example not write $\chi_n$ but rather $-\chi_1-\cdots-\chi_{n-1}$. Also note that the only $\chi$ with $m(\chi)=1$ are those equivalent to $-\chi_i$, so we recover the fact that the $[-\chi_i]$ are the only things in the complement of $\Sigma^1(H_n)$.

This leaves open the question of whether $[\chi]\not\in\Sigma^{m(\chi)}(H_n)$ always holds, which we expect should be true.

\begin{conjecture}\label{conj:houghton_neg}
 With the setup of Theorem~\ref{thrm:houghton_pos}, moreover $[\chi]\not\in\Sigma^{m(\chi)}(H_n)$.
\end{conjecture}

This conjecture holds for low values of $m(\chi)$. When $m(\chi)=1$, without loss of generality $\chi=-\chi_1$, and $[-\chi_1]\not\in\Sigma^1(H_n)$ as mentioned above. When $m(\chi)=2$ (so $n\ge3$), without loss of generality $[\chi]$ lies in the convex hull in $S(H_n)$ of $[-\chi_1]$ and $[-\chi_2]$. Since these are not in $\Sigma^1(H_n)$, \cite[Theorem~A1]{kochloukova02} tells us that $[\chi]$ is not in $\Sigma^2(H_n)$. In general, Conjecture~\ref{conj:houghton_neg} is equivalent to conjecturing that the complement of $\Sigma^m(H_n)$ is the union of all convex hulls of all $\le m$-tuples from the complement of $\Sigma^1(H_n)$; for metabelian groups this is conjectured to always hold, and is called the $\Sigma^m$-Conjecture (see, e.g., \cite[Section~1.3]{bieri10}).

\medskip

We will prove Theorem~\ref{thrm:houghton_pos} by inspecting the proper action of $H_n$ on a proper $\CAT(0)$ cube complex $X_n$. Our reference for $X_n$ is \cite{lee12} (this is a preprint including the author's PhD thesis results). This cube complex was also remarked upon by Brown in \cite{brown87}, though not explicitly constructed. We will not prove everything in the setup here, but will sometimes just cite \cite{lee12}. The vertices of $X_n$ are elements of the monoid $M$ of injective maps $\{1,\dots,n\}\times\N \hookrightarrow \{1,\dots,n\}\times\N$ that are eventually translations. In particular $H_n$ sits in $X_n$ as a discrete set of vertices, namely those maps $\phi$ that are bijective. To describe the higher dimensional cells of $X_n$, we need to discuss $M$ in a bit more detail.

There are $n$ elements of $M$ of particular interest, namely for each $1\le i\le n$ we have a map
$$t_i \colon \{1,\dots,n\}\times\N \to \{1,\dots,n\}\times\N$$
given by sending $(j,x)$ to itself if $j\ne i$ and $(i,x)$ to $(i,x+1)$ for all $x\in\N$. It is clear that for any $\phi\in M$, there exists a product of $t_i$s, say $\tau$, such that $\tau \circ \phi$ is a product of $t_i$s. Here our maps act on the right, so this composition means first do $\tau$, then do $\phi$. Heuristically, $\phi$ acts as translations outside of some finite region $S\subseteq \{1,\dots,n\}\times\N$, so just choose $\tau$ such that the range of $\tau$ lies outside $S$.

Back to defining the higher cells of $X_n$, we now declare that two vertices $\phi,\psi$ share an edge whenever $\phi = t_i \circ \psi$ or $\psi = t_i \circ \phi$ for some $1\le i\le n$. Already we have that $X_n$ is connected, thanks to the discussion in the previous paragraph. Now for $2\le k\le n$, we declare that we have a $k$-cube supported on every set of vertices of the following form: start with a vertex $\phi$, let $K$ be a subset of $\{1,\dots,n\}$ with $|K|=k$, and look at the set of $2^k$ vertices
$$\left\{\left(\prod_{i\in J} t_i\right)\circ\phi \mid J\subseteq K\right\} \text{.}$$
Since the $t_i$ all commute, we do not need to specify an order in which to compose them. These vertices span a $k$-cube in $X_n$. For example, when $n\ge2$ any vertex $\phi$ lies in the square $\{\phi,t_1\circ\phi,t_2\circ\phi, t_1\circ t_2\circ \phi\}$.

It is known that $X_n$ is a $\CAT(0)$ cube complex \cite{lee12}. The group $H_n$ acts on the vertices of $X_n$ via $(\phi)\eta \defeq \phi \circ \eta$, and this extends to an action of $H_n$ on $X_n$. There is an $H_n$-invariant height function $f$ (called $h$ in \cite{lee12}) on the vertices of $X_n$, namely if $\phi\in M$ and $F(\phi) \defeq (\{1,\dots,n\}\times\N) \setminus \image(\phi)$, so $F(\phi)$ is finite, then
$$f(\phi) \defeq |F(\phi)| \text{.}$$
Note that $f(\phi)=0$ if and only if $\phi$ is bijective, i.e., $\phi\in H_n$. It is clear that $f$ is $H_n$-invariant. Also note that for any cube $\sigma$ in $X_n$, there is a unique vertex of $\sigma$ with minimal $f$-value, and so any cube stabilizer is contained in a vertex stabilizer. Vertex stabilizers are finite, since if $\phi \circ \eta = \phi$ then $\eta$ must fix all points outside $F(\phi)$.

In summary, $H_n$ acts properly on the $n$-dimensional proper $\CAT(0)$ cube complex $X_n$. The action is not cocompact, since it is $f$-invariant and $f$ takes infinitely many values, but it is cocompact on $f$-sublevel sets:

\begin{lemma}\label{lem:houghton_cocpt}
 The action of $H_n$ on any $X_n^{p\le f\le q}$ is cocompact.
\end{lemma}

\begin{proof}
 Since $X_n$ is locally compact, we just need to see that $H_n$ is transitive on vertices with the same $f$ value. Let $\phi$ and $\psi$ be vertices with $f(\phi)=f(\psi)$. Let $\alpha\in S_\infty\le H_n$ be any bijection taking $F(\phi)$ bijectively to $F(\psi)$ (since $f(\phi)=f(\psi)$ such a $\alpha$ exists). Define $\eta\in H_n$ via:
 $$(x)\eta \defeq \left\{\begin{array}{ll}
                          (y)\psi & \text{if } x=(y)\phi \\
                          (x)\alpha & \text{if } x\in F(\phi) \text{.}
                         \end{array}\right.$$
 Now $\phi \circ \eta = \psi$ by definition, and $\eta$ clearly eventually acts by translations, so we just need to show $\eta$ is bijective. Note that $\eta$ takes $\image(\phi)$ bijectively to $\image(\psi)$, and also takes $F(\phi)$ bijectively to $F(\psi)$, so indeed $\eta$ is bijective.
\end{proof}

Extending $f$ to a Morse function on $X_n$ (technically the Morse function is $(f,0)$, if we use our definition of Morse function in Definition~\ref{def:morse}), to figure out the higher connectivity properties of the $X_n^{p\le f\le q}$ it suffices to look at $f$-descending links of vertices.

\begin{cit}\cite[Lemma~3.52]{lee12}\label{cit:houghton_f_desc_lk_conn}
 Let $\phi$ be a vertex in $X_n$. If $f(\phi)\ge 2n-1$ then the descending link of $\phi$ is $(n-2)$-connected.
\end{cit}

In particular, Corollary~\ref{cor:morse} says that $X_n^{f\le q}$ is $(n-2)$ connected for $q\ge 2n-2$. Setting $Y\defeq X_n^{f\le 3n-3}$, we have the whole setup of Definition~\ref{def:bnsr}, namely $H_n$ acts properly and cocompactly on the $(n-2)$-connected complex $Y$ (it will become clear later why we use $3n-3$ instead of $2n-2$). Hence to understand $\Sigma^m(H_n)$, we can inspect filtrations of the form $(Y^{t\le \chi})_{t\in\R}$ for $\chi\in\Hom(H_n,\R)$.

We have to explain what $\chi$ means as a function $Y \to \R$. The generating characters $\chi_i$ of $H_n$ measure the length of the eventual translations that every element of $H_n$ must have. Of course vertices of $X_n$ are also functions on $\{1,\dots,n\}\times\N$ that act as eventual translations, so the $\chi_i$ are all naturally defined on vertices of $X_n$, and extend affinely to $X_n$. Note that whereas $\chi_1+\cdots+\chi_n=0$ as characters on $H_n$, now more generally $\chi_1+\cdots+\chi_n=f$ as functions on $X_n$.

Let $\chi=a_1\chi_1+\cdots a_n\chi_n$ be a non-trivial character of $H_n$, so the $a_i$ are not all equal. Up to symmetry assume $a_1\le\cdots\le a_n$. Choose $1\le m(\chi)\le n-1$ maximal with $a_{m(\chi)}\ne a_n$. Since $\chi_1+\cdots+\chi_n=0$ as a character of $H_n$, without loss of generality $a_{m(\chi)+1}=a_n=0$. For instance, instead of $\chi_n$, we consider the equivalent character $-\chi_1-\cdots-\chi_{n-1}$. In general now the first $m(\chi)$ many coefficients of $\chi$ are negative, and all the coefficients from the $(m(\chi)+1)$st one on are zero.

Consider the function $h\defeq (\chi,f)\colon Y \to \R$ with $(\chi,f)$ ordered lexicographically. This is a Morse function, \`a la Definition~\ref{def:morse}, for reasons similar to those in the proof of Proposition~\ref{prop:char_morse} for $F_n$. The key property is that the basis characters vary by $0$, $1$, or $-1$ between adjacent vertices. We now claim that all $h$-ascending links of vertices in $Y$ are $(m(\chi)-2)$-connected, and then Theorem~\ref{thrm:houghton_pos} will follow from Corollary~\ref{cor:morse}.

\begin{lemma}\label{lem:houghton_asc_lk_model}
 Let $\phi$ be a vertex in $Y$. An adjacent vertex $\psi$ is in the $h$-ascending link of $\phi$ if and only if either
 \begin{enumerate}
  \item $\psi=t_i \circ \phi$ for some $m(\chi)+1\le i\le n$, or else
  \item $\phi=t_i \circ \psi$ for some $1\le i\le m(\chi)$.
 \end{enumerate}
\end{lemma}

\begin{proof}
 That $\psi$ is in the link of $\phi$ means there is $1\le i\le n$ such that $\psi=t_i \circ \phi$ or $\phi=t_i \circ \psi$. In the former case, $\psi$ has higher $\chi_i$ and $f$ values than $\phi$ and equal $\chi_j$ values (for $j\ne i$), and in the latter case $\psi$ has lower $\chi_i$ and $f$ values than $\phi$ and equal $\chi_j$ values (for $j\ne i$). Hence in the former case $\psi$ in the $h$-ascending link of $\phi$ if and only if $m(\chi)+1\le i\le n$, since then $\chi$ does not change but $f$ goes up, and in the latter case $\psi$ in the $h$-ascending link of $\phi$ if and only if $1\le i\le m(\chi)$, since then $\chi$ goes up.
\end{proof}

Note that if $\phi=t_i \circ \psi$ and $\psi'=t_j \circ \phi$ ($i\ne j$) then $\psi$ and $\psi'$ share an edge in $\lk\phi$. In particular $\lk^{h\uparrow}_Y \phi$ is a join, of its intersection with $\lk^{f\uparrow}_Y \phi$ and its intersection with $\lk^{f\downarrow}_Y \phi$. Call the former the \emph{ascending up-link} and the latter the \emph{ascending down-link}. The two cases in Lemma~\ref{lem:houghton_asc_lk_model} are thus complete descriptions of the vertices in, respectively, the ascending up-link and ascending down-link.

\begin{proposition}\label{prop:houghton_asc_lk_conn}
 Let $\phi$ be a vertex in $Y$. Then $\lk^{h\uparrow}_Y \phi$ is $(m(\chi)-2)$-connected.
\end{proposition}

\begin{proof}
 We know $0\le f(\phi)\le 3n-3$. First suppose $0\le f(\phi)\le 2n+m(\chi)-3$. The subscripts $i$ for which $t_i \circ \phi$ is ascending are those satisfying $m(\chi)+1\le i\le n$, so there are $n-m(\chi)$ of them, and since $(2n+m(\chi)-3) + (n-m(\chi)) = 3n-3$ we have in this case that the entire $f$-ascending link of $\phi$ in $X_n$ is contained in $Y$. This tells us that the ascending up-link of $\phi$ consists of the $(n-m(\chi)-1)$-simplex $\{t_{m(\chi)+1},\dots,t_n\}$, which is contractible, and hence $\lk^{h\uparrow}_Y \phi$ is contractible.
 
 Now suppose $2n+m(\chi)-2 \le f(\phi)\le 3n-3$, so $Y$ does not contain the entire ascending up-link of $\phi$ in $X_n$, but rather only its $(3n-f(\phi)-4)$-skeleton. This is the $(3n-f(\phi)-4)$-skeleton of an $(n-m(\chi)-1)$-simplex, so it is $(3n-f(\phi)-5)$-connected. Since $f(\phi)\ge 2n+m(\chi)-2 \ge 2n-1$ though, in this case we have that the entire ascending down-link of $\phi$ in $X_n$ is contained in $Y$. Lemma~\ref{lem:houghton_asc_lk_model} tells us that this $h$-ascending down-link is isomorphic to the $f$-descending link in $X_{m(\chi)}$ of a vertex with $f$ value equal to $f(\phi)$. Since $f(\phi)\ge 2n+m(\chi)-2 \ge 2m(\chi)-1$, this is $(m(\chi)-2)$-connected by Citation~\ref{cit:houghton_f_desc_lk_conn}. In this case, taking the join, we see that $\lk^{h\uparrow}_Y \phi$ is $((3n-f(\phi)-4)+(m(\chi)-1))$-connected, and hence $(3n-f(\phi)+m(\chi)-5)$-connected. The result now follows since $f(\phi)\le 3n-3$.
\end{proof}

\begin{proof}[Proof of Theorem~\ref{thrm:houghton_pos}]
 The superlevel sets $Y^{t\le \chi}$ are all $(m(\chi)-2)$-connected by Corollary~\ref{cor:morse} and Proposition~\ref{prop:houghton_asc_lk_conn}, so by Definition~\ref{def:bnsr}, $[\chi]\in\Sigma^{m(\chi)-1}(H_n)$.
\end{proof}

\medskip

As for negative properties, i.e., Conjecture~\ref{conj:houghton_neg}, it is difficult in general to tell using Morse theory that a filtration is \emph{not} essentially $(m-1)$-connected. Even if we know the ascending link of a vertex is not $(m-1)$-connected, we do not know whether gluing in that vertex served to kill a pre-existing $(m-1)$-sphere, or served to create a new $m$-sphere. For example if a vertex's ascending link is two points, we do not know whether gluing in that vertex connects up two previous disconnected components, or creates a loop. This is basically what makes it so difficult to prove that character classes lie in $S(G)\setminus\Sigma^m(G)$; for example even when $G$ is metabelian this problem remains open in general.

As a remark, to show that $(Y^{t\le \chi})_{t\in\R}$ is not essentially $(m(\chi)-1)$-connected, it suffices to prove that $Y^{0\le\chi}$ is not $(m(\chi)-1)$-connected, by tricks for negative properties discussed in \cite{witzel15}. Also, thanks to how we have realized $\chi$ as a linear combination of the $\chi_i$ using non-positive coefficients, for any vertex $x\in X_n^{0\le\chi}$ the whole $f$-descending link of $x$ lies in $X_n^{0\le\chi}$. Hence $Y^{0\le\chi}$ is $(m(\chi)-1)$-connected if and only if $X_n^{0\le\chi}$ is. This reduces Conjecture~\ref{conj:houghton_neg} to proving that $X_n^{0\le\chi}$ is not $(m(\chi)-1)$-connected, but this is still a hard problem when $m(\chi)>2$, beyond the scope of our present techniques.

As a final remark, one typical trick for deducing negative properties is finding a retract onto a more manageable quotient with negative properties. However, for $H_n$, every proper quotient $Q$ has $\Sigma^1(Q)=S(Q)$, and so it seems very unlikely this trick could work. To see this fact about such $Q$, we note that for $n\ge3$, $[H_n,H_n]=S_\infty$, the infinite symmetric group, and the second derived subgroup $H_n^{(2)}$ is $A_\infty$, the infinite alternating group. One can check that every non-trivial normal subgroup of $H_n$ contains $A_\infty$, so any $Q$ as above is a quotient of $H_n/A_\infty$. But every kernel of a character on $H_n$ becomes finitely generated when taken mod $A_\infty$, so indeed $\Sigma^1(Q)=S(Q)$.

\providecommand{\bysame}{\leavevmode\hbox to3em{\hrulefill}\thinspace}
\providecommand{\MR}{\relax\ifhmode\unskip\space\fi MR }
\providecommand{\MRhref}[2]{%
  \href{http://www.ams.org/mathscinet-getitem?mr=#1}{#2}
}
\providecommand{\href}[2]{#2}

\end{document}